\documentclass[leqno, letterpaper, 11pt]{article}
    \usepackage{amsmath}
    \usepackage{amsfonts} 
    \usepackage{fullpage} 
    \usepackage{graphicx}
	\usepackage{amssymb}
	\usepackage{mathrsfs}
	\usepackage{hyperref}
	\usepackage{cleveref}
	\usepackage{float}
	\usepackage{epstopdf}
	\usepackage{tikz-cd}
	\usepackage{pgf, tikz}
	\usepackage{thmtools}
	\usepackage{bm}
	\usepackage[numbers]{natbib}
	\usepackage{multicol}
\usepackage{thm-restate}
\usetikzlibrary{arrows, automata}
\usepackage[T1]{fontenc}
\usepackage[font=small,labelfont=bf,tableposition=top]{caption}
 
\DeclareCaptionLabelFormat{andtable}{#1~#2  \&  \tablename~\thetable}

\usepackage{amssymb}
\usepackage{lineno}

\usepackage{graphicx}
\usepackage{subfig}
\usepackage{amssymb}
\usepackage{amsmath,amsfonts,amssymb,amsthm}
\usepackage{graphicx,epsfig,setspace}
\onehalfspace

\usepackage{etoolbox}
\patchcmd{\thebibliography}{\section*{\refname}}{}{}{}

\newcommand{\Z}{\mathbb{Z}}
\newcommand{\cN}{\mathcal{N}}

\newcommand{\ord}{\text{\rm ord}}
\newcommand{\ub}{\text{\tt ub}}
\newcommand{\bdot}{\circ}
\newcommand{\arb}{\diamond}
\newcommand\blfootnote[1]{%
  \begingroup
  \renewcommand\thefootnote{}\footnote{#1}%
  \addtocounter{footnote}{-1}%
  \endgroup
}

\theoremstyle{plain}
\newtheorem{thm}{Theorem}

\newtheorem{lem}{Lemma}[section]
\newtheorem{cor}[lem]{Corollary}
\newtheorem{prop}[lem]{Proposition}

\theoremstyle{definition}
\newtheorem{defi}[lem]{Definition}

\newtheorem{ques}[lem]{Question}

 \newtheorem*{prob*}{Problem}

\theoremstyle{remark}

\begin{document}

\title{Maximal temporal period of a periodic solution 
generated by a one-dimensional 
cellular automaton}
 
\author{{\sc Janko Gravner} and {\sc Xiaochen Liu}\\
Department of Mathematics\\University of California\\Davis, CA 95616\\{\tt gravner{@}math.ucdavis.edu, xchliu{@}math.ucdavis.edu}}
\maketitle
\bibliographystyle{plain}

\begin{abstract}
We study one-dimensional cellular automata evolutions with both temporal and spatial periodicity.
The main objective is to investigate the longest temporal periods among all two-neighbor rules, with a fixed spatial period $\sigma$ and number of states $n$.
When $\sigma = 2, 3, 4$, or $6$, and we restrict the rules to be additive, the longest period can be expressed as the exponent of the multiplicative group of an appropriate ring. 
We also construct non-additive rules with  temporal period on the same order as the trivial upper bound $n^\sigma$.
Experimental results, open problems, and possible extensions of our results are also discussed.
\end{abstract}

\blfootnote{\emph{Keywords}: Cellular automaton, exponent of a multiplicative group, periodic 
solution.}
\blfootnote{AMS MSC 2010:  37B15, 68Q80.}

\section{Introduction}

We continue our study of periodic solutions of one-dimensional  $n$-state cellular automata (CA) from \cite{gl1} and \cite{gl2}.
In those two 
papers, we
assumed a fixed spatial period $\sigma$ and discussed the temporal periods for randomly selected rules. 
In the present paper, we instead investigate the analogous extremal questions.

We refer to elements of $\Z$ as \textbf{sites}, and, for a fixed 
$n\ge 2$,  to elements 
of $\mathbb{Z}_n = \{0, 1, \dots n - 1\}$ as \textbf{states} or \textbf{colors}. 
A (one-dimensional) \textbf{spatial configuration} is a coloring of sites, that is, a 
map $\xi:\Z\to \mathbb{Z}_n$.
A one-dimensional CA is a spatially and temporally discrete dynamical system of evolving spatial configurations $\xi_t$, 
$t \in \mathbb{Z}_+ = \{0, 1, \dots\}$.
In general, the dynamics of such CA is determined by a \textbf{neighborhood} $\cN\subset \Z$ of $r\ge 1$ sites and by its \textbf{(local) rule}, 
which is
a function $f: \mathbb{Z}_n^r \to \mathbb{Z}_n$.
In this paper, as in \cite{gl1}, we assume the simplest non-trivial case when $r = 2$ and $\cN=\{-1,0\}$. Thus, 
the spatial configuration updates from $\xi_t$ to $\xi_{t + 1}$ 
using a rule $f: \mathbb{Z}_n^2 \to \mathbb{Z}_n$ as follows:
$$\xi_{t + 1} (x) = f(\xi_t(x - 1), \xi_t(x)),$$
for all $x \in \mathbb{Z}$. 
We sometimes write $c_0\underline{c_1} \mapsto c_2$ instead of $f(c_0, c_1) = c_2$.
A rule $f$ is  \textbf{additive} if it commutes with sitewise 
addition modulo $n$ or, equivalently, if there exist $a,b\in \Z_n$ so that $f(c_0, c_1) = bc_0 + ac_1 \mod n$, for all $c_0, c_1 \in \mathbb{Z}_n$.
Once $\xi_0$ is specified, the rule 
determines the CA \textbf{trajectory} $\xi_0, \xi_1,\ldots$, 
which we also identify with its space-time assignment $\mathbb{Z} \times \mathbb{Z}_+ \to \mathbb{Z}_n$, given by
$(x, t) \mapsto\xi_{t}(x)$. 

We focus on CA whose trajectories are periodic in both directions. 
We call a spatial configuration $\xi$ \textbf{periodic} if $\xi(x) = \xi(x + \sigma)$, for all $x \in \mathbb{Z}$ and a $\sigma > 0$.
If $\sigma$ is the smallest such number, we call $\sigma$ the \textbf{spatial period} of $\xi$.
It is clear that, if $\xi_t$ is periodic with period $\sigma$, then $\xi_{t + 1}$ is also periodic with a period that 
divides $\sigma$. Observe also that, if $\xi_0$ is periodic 
with period $\sigma$, we may view the evolution of the CA as 
the sequence of colorings of $\{0,1,\ldots, \sigma-1\}$, 
with periodic boundary, as in \cite{martin1984algebraic}. 
If, for some $\ell$,  $\xi_\ell$ is periodic with period $\sigma$, and $\tau\ge 1$ is the smallest integer such that $\xi_{\ell +\tau} = \xi_{\ell}$, then we call $\xi_\ell$ a \textbf{periodic solution} (\textbf{PS}) of rule $f$, with \textbf{temporal period} $\tau$ and spatial period $\sigma$. We can specify a particular PS by 
any $\sigma$ contiguous states $\xi_j(x)\xi_j(x + 1) \dots \xi_j(x + \sigma - 1)$, for any $x \in \mathbb{Z}$ and $\ell\le j< \ell+\tau$.
See Figure \ref{figure: PS example} for an example.
In this figure, $n = 3$ and $f$ is the additive rule given by 
$f(c_0, c_1) = c_0 + c_1$ for all $c_0, c_1 \in \mathbb{Z}_3$.
This PS has spatial period $\sigma = 4$ and 
temporal period $\tau = 8$, and can be specified by 
any $\sigma = 4$ contiguous states, say $2101$. The 
temporal period $\tau=8$ is the 
largest of all additive rules with $\sigma=4$ and $n=3$.
 
\begin{figure}[h] 
    \centering
    \includegraphics[scale = 0.25, trim = 5cm 1cm 3cm 1cm, clip]{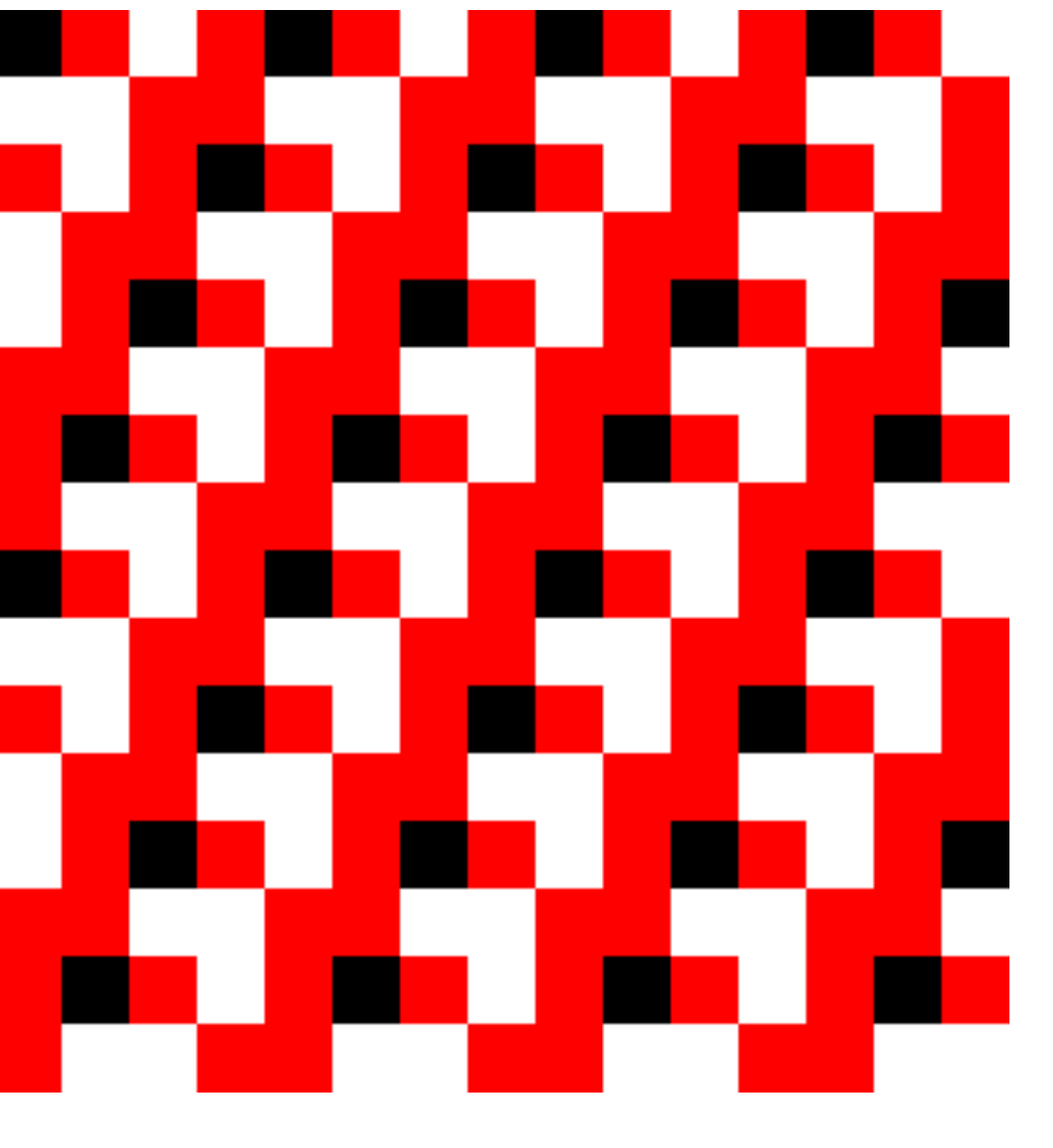}\\
    \caption{
    A $16\times 16=4\sigma\times 2\tau$ piece of trajectory of a PS of a 3-state additive rule.
    States 0, 1, and 2 are represented by white, red and black cells, respectively. The time axis is oriented downward, as is customary.
}
    \label{figure: PS example}
  \end{figure}

For an $n$-state rule $f$ and $\sigma\ge 1$, we 
let $X_{\sigma, n}(f)$ and $Y_{\sigma, n}(f)$ be, respectively, the largest and smallest temporal periods of PS, with spatial period $\sigma$, of the rule $f$. When $f$ is selected uniformly at random, $X_{\sigma, n}$ and $Y_{\sigma, n}$ become random variables, which we investigated in \cite{gl1, gl2}. 
In \cite{gl1}, we proved that the smallest temporal period $Y_{\sigma, n}$ converges in distribution to a non-trivial limit, as $n \to \infty$; in particular, it is stochastically bounded. 
By contrast, the longest temporal period $X_{\sigma, n}$ 
is expected to be on the order $n^{\sigma/2}$. 
We prove this in \cite{gl2} in the more general $r$-neighbor setting, 
but for our methods to work we are forced to assume 
that $\sigma \le r$. Then, $X_{\sigma, n}/n^{\sigma/2}$ converges in distribution to a non-trivial limit, 
as $n\to\infty$. The case 
$\sigma>r$ is still open, even in our present case $r=2$. 

Instead of their typical size, this paper explores the extremal values of quantities $X_{\sigma, n}(f)$ and $Y_{\sigma, n}(f)$. 
It is clear that $\min_f Y_{\sigma, n}(f)  = \min_fX_{\sigma, n}(f) = 1$, as the minima are 
attained by the identity $n$-state rule, 
i.e., the rule $f$ given by $f(c_0, c_1) = c_1$, for all $c_0, c_1 \in \mathbb{Z}_n$. 
We therefore focus on 
\begin{equation}\label{extq}
\max_fY_{\sigma, n}(f)\text{ and }\max_fX_{\sigma, n}(f),
\end{equation}
the largest among the shortest and longest temporal periods of a PS with spatial period $\sigma$ and $n$ states.
Let
$T(\sigma, n)$ be the number of aperiodic 
length-$\sigma$ words from alphabet $\mathbb{Z}_n$, 
that is, words that cannot be written as repetition of a subword. Then it is clear that, for all $n$ state rules $f$,
$1 \le Y_{\sigma, n}(f)  \le X_{\sigma, n}(f) \le T(\sigma, n)$. 
We also have the following counting result.

\begin{lem}\label{lemma-T}
The number of aperiodic length-$\sigma$ word from alphabet $\mathbb{Z}_n$ is 
\begin{align*}
T(\sigma, n) & = 
\sum_{d\bigm| \sigma} n^d  \mu\left(\frac{\sigma}{d}\right)\\
& = 
\begin{cases}
n^\sigma - n^{\sigma/2} + o(n^{\sigma/2}), & \text{if } \sigma \text{ is even}\\
n^\sigma + o(n^{\sigma/2}), & \text{if } \sigma \text{ is odd}
\end{cases},
\end{align*}
where $\mu(\cdot)$ is the M\"{o}bius function.
\end{lem}
\begin{proof}
See \cite{cattell2000fast}.
\end{proof}
	
For $\sigma = 1$ and any $n$, it is easy to find a rule $f$ with $Y_{1, n}(f) = X_{1, n}(f) = n = T(1, n)$; for example, any rule $f$ satisfying $f(a, a) = \phi(a)$, where $\phi$ is any permutation on $\mathbb{Z}_n$ of order $n$, would do.  
For $\sigma = 2$, viewing evolution on $\{0,1\}$ with periodic boundary, a unique CA with temporal period $\binom {n}{2}$ goes through all length-$2$ configuration $ab$, with $a < b \in \mathbb{Z}_n$. 
For instance, when $n = 3$, the evolution
$$
\begin{matrix}
0 & 1\\
0 & 2\\
1 & 2
\end{matrix}
$$
defines a rule with $0\underline{1} \mapsto 2$,
$1\underline{0} \mapsto 0$,
$0\underline{2} \mapsto 2$,
$2\underline{0} \mapsto 1$,
$1\underline{2} \mapsto 1$
and
$2\underline{1} \mapsto 0$.
Switching the last two values of $f$ extends the PS to
$$
\begin{matrix}
0 & 1\\
0 & 2\\
1 & 2\\
1 & 0\\
2 & 0\\
2 & 1
\end{matrix}
\quad ,$$ 
which has temporal period $6 = 3^2 - 3 = T(2, 3)$.
It is clear that this construction works for all $n$ and gives   
$Y_{2, n}(f) = X_{2, n}(f) = n^2 - n = T(2, n)$. 

Even for $\sigma \ge 3$, 
it is not obvious what 
the extremal values (\ref{extq})
are, whether they are equal, or whether 
the upper bound $T(3, n)$ 
can always be attained. One of our main results is that 
$\max_f Y_{\sigma, n}(f)= \Theta(n^\sigma)$, matching the
order of $T(\sigma, n)$ given by Lemma~\ref{lemma-T}.

\begin{thm} \label{theorem: main of general rule}
Fix an arbitrary $\sigma > 0$.
For $n\ge N(\sigma)$, there exists an $n$-state CA rule $f$ such that 
$X_{\sigma, n}(f) = Y_{\sigma, n}(f) \ge C(\sigma)n^\sigma$, where $N(\sigma)$ and $C(\sigma)$ are constants depending only on $\sigma$. 
\end{thm}

To alleviate the difficulties in computing the extremal quantities (\ref{extq}), one may try to restrict the set of rules $f$. The most natural such 
restriction are 
the additive rules, which exploit the algebraic structure of 
the states and enable the use of algebraic tools \cite{martin1984algebraic, jen1, jen2} .
We denote by $\mathcal{A}_n$ the set of $n$-state additive rules and let 
$$\pi_\sigma(n) = \max_{f \in \mathcal{A}_n} X_{\sigma, n}(f).$$
It follows from \cite{martin1984algebraic} that $\pi_\sigma(n) \le n^{\sigma - 1}$ (see Corollary \ref{corollary: main}), and therefore by Theorem~\ref{theorem: main of general rule} the maximal period of additive rules 
is at least by 
one power of $n$ smaller than that of non-additive 
rules.
Furthermore, for $\pi_\sigma(n)$ and $\sigma \in \{2, 3, 4, 6\}$, we 
are able to give an explicit formula for $\pi_\sigma(n)$. 
Let $\lambda_\sigma(n)$ be the exponents of multiplicative group of $\Z_n$ when $\sigma=2$, Eisenstein integers modulo $n$ when $\sigma=3$, and Gaussian integers modulo $n$ when $\sigma =4$. Then $\pi_\sigma$ is related to $\lambda_\sigma$ as follows. 

\begin{thm} \label{theorem: main of additive rule}
For $\sigma = 2, 3$, $\pi_\sigma(n) = \lambda_\sigma(\sigma n)$, for all $n \ge  2$. Moreover, $\pi_4(2)=4$ and 
$\pi_4(n) = \lambda_4(n)$, for all $n \ge 3$. Finally,  
$\pi_6(n) = \lambda_3(6n)$, for all $n \ge 2$.
\end{thm}

This theorem, and Lemmas~\ref{lemma: lcm of exponent}--\ref{lemma: lambda_4}, give
the promised explicit expressions for the four $\pi_\sigma(n)$.
It is tempting to conjecture that a variant of Theorem~\ref{theorem: main of additive rule} holds for all $\sigma$, 
with a suitable definition of $\lambda_\sigma$ for Kummer
ring $\Z_n(\zeta)$, where $\zeta$ is the $\sigma$'th 
root of unity. However, this remains unclear as $\zeta$ is quadratic only for 
$\sigma=3,4,6$, and this fact plays a crucial role in 
our arguments. 


We now give a brief review of the previous literature on large
temporal periods of PS.
The fundational work on the temporal periods of additive CA is certainly \cite{martin1984algebraic}. Various 
recursive relations and upper bounds given in this paper are 
very useful, and indeed are utilized  
in the proof of Theorem~\ref{theorem: main of additive rule} in  Section \ref{section: additive}. Like the present paper, \cite{martin1984algebraic}, and its notable successors such as \cite{jen1, jen2}, study CA on finite intervals 
with periodic boundary. This choice is important, as
results with other type of boundaries yield substantially different results. The paper \cite{chang1997maximum}, for example, investigates the maximal length of temporal periods of binary CA under null boundary condition, and demonstrates that the maximal length $2^\sigma - 1$ can be obtained by additive rules, for any $\sigma > 0$.
In \cite{adak2018there}, the authors address the same question  for non-additive CA, and 
show that the maximal length can also be obtained, if the rule
is allowed to be non-uniform among sites.
Works that investigate additive rules and their temporal periods also include \cite{guan1986exact}, \cite{pries1986group}, 
\cite{thomas2006characteristic}, and \cite{misiurewicz2006iterations}.

The rest of the paper is organized as follows.
In Section \ref{section: additive} we address additive rules and prove Theorem \ref{theorem: main of additive rule}. We relegate 
a result on multiplicative group structure of Eisenstein numbers modulo $n$, which is needed for $\sigma=3,6$, to the Appendix
at the end of the paper. 
In Section \ref{section: general}, we prove Theorem \ref{theorem: main of general rule} through explicit construction.
Finally, in Section \ref{section: conclusion}, we present several simulation results and propose a number of resulting conjectures.

\section{Longest temporal periods of additive rules} \label{section: additive}
In this section, we investigate the longest temporal period that an additive rule is able to generate, for a fixed spatial period $\sigma$. 

\subsection{Definitions and preliminary results}
\label{section: additive-def}
We write a configuration $\xi_t$ on the integer interval $[0,\sigma-1]$ with periodic boundary as
$c_0^{(t)}c_1^{(t)}\dots c_{\sigma - 1}^{(t)}$, where $c_j^{(t)} \in \mathbb{Z}_n$, for $j = 0, 1, \ldots, \sigma - 1$, or, 
equivalenty, by the polynomial of degree $\sigma-1$\cite{martin1984algebraic}
$$
L^{(t)}(x) = \sum_{j = 0}^{\sigma - 1} c_j^{(t)}x^j.
$$
An additive rule $f$ such that $f(c_0, c_1) = bc_0 + ac_1$, for $a, b\in \mathbb{Z}_n$ is characterized by the polynomial $T(x) = a + bx$, and its evolution as polynomial multiplication:
$$
L^{(t + 1)}(x) = T(x)L^{(t)}(x),
$$
in the quotient ring of polynomials $\mathbb{Z}_n[x]$ modulo the ideal generated by the polynomial $x^\sigma  -1$, to  implement the periodic 
boundary condition. In this section, we will use $T(x)$, for some fixed $a$ and $b$, to specify an additive CA, in place of the rule $f$. 

As a result, a PS generated by the additive rule $T(x) = a + bx$ with temporal period $\tau$ and spatial period $\sigma$ satisfies
$$
T^{\tau}(x)L^{(\ell)}(x) = L^{(\ell)}(x), \text{ in }\mathbb{Z}_n[x]/(x^\sigma - 1).
$$ 
We are interested in the longest temporal period with a fixed spatial period $\sigma$.
For general CA, this task requires the examination of the longest cycle in the configuration directed graph \cite{gl2}, 
which encapsulates information from all initial configurations. 
For linear rules, however, the following simple proposition from \cite{martin1984algebraic} reduces the set of relevant
initial configurations to a singleton. 

\begin{prop} (Lemma 3.4 in \cite{martin1984algebraic}) \label{proposition: longest period starts with 1} Fix an additive CA and a $\sigma\ge 1$. 
The temporal period of any PS with the spatial period $\sigma$   divides the temporal period resulting from	the initial configuration $1\,0^{\sigma -1}$ ($1$ followed by $\sigma-1$ $0$s), represented by the constant polynomial $1$. 
\end{prop}

Therefore, we may define the longest temporal period 
$\Pi_{\sigma}(a, b; n)$ of an additive rule $T(x) = a + bx$, 
as the smallest $k$, such that 
$$(a + bx)^{k + \ell} = (a + bx)^\ell, \text{ in } \mathbb{Z}_n[x]/(x^\sigma - 1),$$
for some $\ell\ge 0$. We will 
refer to $\Pi_{\sigma}(a, b; n)$ 
as simply the \textbf{period} of $T(x)$. 
The largest period is thus
$$\pi_\sigma(n)= \max_{a, b\in\mathbb{Z}_n} \Pi_\sigma(a, b; n).$$

We use the standard notation $\mathbb{Z}_n[i]$ (where $i = \sqrt{-1}$) and $\mathbb{Z}_n[\omega]$ (where $\omega = e^{2\pi i/3}$) for Gaussian integers modulo $n$ and Eisenstein integers modulo $n$.
 
For a finite ring $R$ with unity, we denote by $R^\times$ its 
multiplicative group and, define the 
(multiplicative) \textbf{order} $\ord(x)$ for any $x\in R$ to be 
the smallest integer $k$ so that $x^k=1$ if $x\in R^\times$, 
and let $\ord(x)=1$ otherwise. 
Note that this is the standard definition when $x\in R^\times$. Recall that 
\begin{equation*}
\begin{aligned}
&\mathbb{Z}_n^\times = \{a: \gcd(a, n) = 1\},\\
&\mathbb{Z}_n[i]^\times = \{a + bi: a, b \in \mathbb{Z}_n, \gcd(a^2 + b^2, n) = 1\},\\
&\mathbb{Z}_n[\omega]^\times = \{a + b\omega: a, b \in \mathbb{Z}_n, \gcd(a^2 + b^2 - ab, n) = 1\}.
\end{aligned}
\end{equation*} 
Then we define 
\begin{equation}
\begin{aligned}
&\Lambda_2(a, b; n) =\ord(a+b) \text{ in } \Z_n, \\
&\Lambda_3(a, b; n) = \ord(a+b\omega) \text{ in } \Z_n[\omega],
\\
&\Lambda_4(a, b; n) = \ord(a+bi) \text{ in } \Z_n[i].
\end{aligned}
\end{equation}
%
Furthermore, we let 
$$\lambda_\sigma(n) = \max_{a, b \in\mathbb{Z}_n}\Lambda_\sigma(a, b; n), 
$$ 
for $\sigma = 2$, $3$, and $4$, be the 
exponents of the multiplicative groups $\mathbb{Z}_n^\times$,  
$\mathbb{Z}_n[\omega]^\times$, and $\mathbb{Z}_n[i]^\times$.   
In Section \ref{section: exponent of groups}, we obtain explicit 
formulas for $\lambda_\sigma(n)$ for these three $\sigma$'s.


In the sequel, we will 
use $p$, and $p_1, p_2 \dots$ to denote prime numbers; 
for an arbitrary $n$, we write its prime decomposition as $n = p_1^{m_1}\dots p_k^{m_k}$ or as $n = 2^{m_2}3^{m_3} \dots p^{m_p}$.
When $p \nmid \sigma$, we use $\text{ord}_\sigma(p)$ to denote the order of $p$ in $\Z_\sigma$. 
We now list several useful results from \cite{martin1984algebraic}.

\begin{prop} \label{proposition: MOW 4.3}
(Lemma 4.3 in \cite{martin1984algebraic}) If $p \bigm| \sigma$, then $\Pi_\sigma\left(a, b; p\right) \bigm| p\Pi_{\sigma/p}\left(a, b; p\right)$.
\end{prop}

\begin{prop} \label{proposition: MOW 4.1}
(Theorem 4.1 and (B.8) in \cite{martin1984algebraic}) If $p \nmid \sigma$ and $\sigma\ge 2$, then 
$$\Pi_{\sigma}\left(a, b; p\right) \bigm| \left(p^{\text{\rm ord}_\sigma(p)} - 1\right)$$
 and $\text{\rm ord}_\sigma(p) \le \sigma - 1$. Furthermore, 
$\Pi_{1}\left(a, b; p\right) \bigm| \left(p - 1\right)$.
\end{prop}

\begin{prop} \label{proposition: MOW 4.4}
(Theorem 4.4 in \cite{martin1984algebraic}) For $n = p_1^{m_1}\dots p_k^{m_k}$, we have
$$\Pi_{\sigma}(a, b; n) = \text{\rm lcm} \left(
\Pi_\sigma(a, b; p_1^{m_1}), \dots,
\Pi_\sigma(a, b; p_k^{m_k}) \right).$$
\end{prop}

\begin{prop} \label{proposition: MOW 4.5}
(Theorem 4.5 in \cite{martin1984algebraic}) Let $m\ge 2$ be an integer. Then
$\Pi_\sigma\left(a, b; p^{m}\right)$ either equals $p\Pi_\sigma\left(a, b; p^{m - 1}\right)$ or $\Pi_\sigma\left(a, b; p^{m - 1}\right)$.
\end{prop}


As a consequence of the above results, we obtain the  
following upper bound. 

\begin{cor} \label{corollary: main}
Let $\sigma \ge 2$, then $\max_{f \in \mathcal{A}_n}X_{\sigma, n}(f) \le n^{\sigma - 1}$, for all $n \in \mathbb{N}.$
\end{cor}
\begin{proof}
Let $n = p_1^{m_1} \dots p_k^{m_k}$ be the prime decomposition of $n$. For every $j=1,\ldots, k$ write $\sigma = p_j^{n_j} \sigma_j$, where $n_j\ge 0$ and  $\sigma_j$ is such that $p_j \nmid \sigma_j$. Let $\epsilon_j=1$ if $\sigma_j=1$, and 
$\epsilon_j=0$ otherwise.
For any $a, b\in \mathbb{Z}_n$, 
\begin{align*}
\Pi_\sigma(a, b; n) 
& = \text{lcm}\left(\Pi_\sigma(a, b; p_1^{m_1}),\dots, \Pi_\sigma(a, b; p_k^{m_k})\right) \quad (\text{Proposition \ref{proposition: MOW 4.4}}) \\
& \le 
\prod_{j=1}^k   
p_j^{m_j - 1} \Pi_\sigma(a, b; p_j) \quad (\text{Proposition \ref{proposition: MOW 4.5}})\\
& \le 
\prod_{j=1}^k p_j^{m_j+n_j+\sigma_j-2} (p_j-1)^{\epsilon_j}
\quad (\text{Propositions~\ref{proposition: MOW 4.3} and~\ref{proposition: MOW 4.1}})\\
&  \le 
\prod_{j=1}^k p_j^{m_j(\sigma-1)}=n^{\sigma-1}, 
\end{align*}
provided that the inequality  
\begin{equation}\label{main-cor-eq1}
m_j+n_j+\sigma_j-2\le m_j(p_j^{n_j}\sigma_j-1)
\end{equation}
holds when either $\sigma_j\ge 2$ or $p_j=2$, and the inequality
\begin{equation}\label{main-cor-eq2}
m_j+n_j+\sigma_j-1\le m_j(p_j^{n_j}\sigma_j-1)
\end{equation}
holds when $\sigma_j=1$ and $p_j\ge 3$. 

Note that $\sigma_j=1$ implies that $n_j\ge 1$. 
Next, observe that $p_j^{n_j}\ge 2^{n_j}\ge n_j+1$. Assume 
first that 
$\sigma_j\ge 2$. Then we have $m_jp_j^{n_j}\sigma_j\ge 
m_j(n_j+1)\sigma_j\ge n_j\sigma_j +2m_j$. Moreover, if $n_j\ge 1$, then
$n_j\sigma_j-n_j-\sigma_j+1=(n_j-1)(\sigma_j-1)\ge 0$ 
and so (\ref{main-cor-eq1}) holds. If $n_j=0$, then 
(\ref{main-cor-eq1}) reduces to $\sigma_j-2\le m_j(\sigma_j-2)$, 
which again holds. Next 
we assume that $\sigma_j=1$ and $p_j=2$. Then 
(\ref{main-cor-eq1}) follows from $m_j+n_j-1\le m_jn_j$. 
Finally, assume that $\sigma_j=1$  and $p_j\ge 3$. Then 
the inequality (\ref{main-cor-eq2}) follows from 
$n_j\le 3^{n_j}-2$. The equalities (\ref{main-cor-eq1})and 
(\ref{main-cor-eq2}) are thus established and the proof
completed.
\end{proof}

\subsection{Exponents of the multiplicative groups} 
\label{section: exponent of groups}
In this section, we find formulas for $\lambda_\sigma(n)$,  
$\sigma = 2$, $3$, and $4$, i.e., the exponents of multiplicative groups $\mathbb{Z}_n^\times$, $\mathbb{Z}_n[\omega]^\times$, and $\mathbb{Z}_n[i]^\times$.

\begin{lem} \label{lemma: lcm of exponent}
For $\sigma = 2, 3$ and $4$,
$$\lambda_\sigma(n) = \text{\rm lcm} (\lambda_\sigma(p_1^{m_1}), \dots, \lambda_\sigma(p_k^{m_k})).$$
\end{lem}
\begin{proof}
By the Chinese Remainder Theorem, $\mathbb{Z}_n^\times$ (respectively, $\mathbb{Z}_n[\omega]^\times$, $\mathbb{Z}_n[i]^\times$)
is isomorphic to the direct product of the $k$ groups 
 $\mathbb{Z}_{p_j^{m_j}}^\times$ (respectively, $\mathbb{Z}_{p_j^{m_j}}[\omega]^\times$,  $\mathbb{Z}_{p_j^{m_j}}[i]^\times$), 
 $j=1,\ldots,k$. 
\end{proof}
 
To find $\lambda_\sigma(n)$, it therefore suffices to find the formulas for $\lambda_\sigma(p^{m})$ for prime $p$.
For $\sigma = 2$, $\lambda_2$ is known as the Carmichael function, which is given by the following explicit formula.
\begin{lem}\label{lemma: lambda_2}
For $m \ge 1$ and $p$ prime,
$$
\lambda_2(p^m) = 
\begin{cases}
2^{m - 1}, & \text{ if } p = 2 \text{ and } m \le 2\\
2^{m - 2}, & \text{ if } p = 2 \text{ and } m \ge 3\\
p^{m - 1}(p - 1), & \text{ if } p > 2
\end{cases}.
$$
\end{lem}
\begin{proof}
See \cite{carmichael1910note}.
\end{proof}

The results for $\lambda_3$ and $\lambda_4$ follow 
from the classification of the two multiplicative groups. 
For $\Z_{p^m}[i]^\times$, this task was accomplished 
in \cite{allan2008classification}, while for 
$\Z_{p^m}[\omega]^\times$ we relegate the similar argument  
to the Appendix.

\begin{lem} \label{lemma: lambda_3}
For $m\ge 1$ and $p$ prime,
$$
\lambda_3(p^m) = 
\begin{cases}
6, & \text{ if } p = 3 \text{ and } m = 1\\
2\cdot 3^{m - 1}, & \text{ if } p = 3 \text{ and } m \ge 2\\
p^{m - 1}(p - 1), & \text{ if } p = 1 \mod 3\\
p^{m - 1}(p^2 - 1), & \text{ if }p = 2 \mod 3
\end{cases}.
$$
\end{lem}

\begin{proof} The claim follows from Theorem~\ref{theorem: eisenstein} in the Appendix.
\end{proof}

\begin{lem}\label{lemma: lambda_4}
For $m\ge 1$ and $p$ prime,

$$
\lambda_4(p^m) = 
\begin{cases}
2^m, & \text{ if } p = 2 \text{ and } m \le 2\\
2^{m - 1}, & \text{ if } p = 2 \text{ and } m \ge 3\\
p^{m - 1}(p - 1), & \text{ if } p = 1 \mod 4\\
p^{m - 1}(p^2 - 1), & \text{ if }p = 3 \mod 4
\end{cases}.
$$
\end{lem}
\begin{proof}
By \cite{allan2008classification}, we have
$$
\mathbb{Z}_p[i]^\times \cong
\begin{cases}
\mathbb{Z}_2, & \text{ if } p = 2\\
\mathbb{Z}_{p - 1} \times \mathbb{Z}_{p - 1}, & \text{ if } p = 1 \mod 4\\
\mathbb{Z}_{p^2 - 1}, & \text{ if } p = 3 \mod 4
\end{cases}
$$
and
$$
\mathbb{Z}_{p^m}[i]^\times \cong
\begin{cases}
\mathbb{Z}_{p^{m - 1}} \times \mathbb{Z}_{p^{m - 2}} \times \mathbb{Z}_4, & \text{ if } p = 2 \text{ and } m \ge 2\\
\mathbb{Z}_{p^{m - 1}} \times \mathbb{Z}_{p^{m - 1}} \times \mathbb{Z}_p[i]^\times, & \text{ if } p\neq 2
\end{cases}.
$$
The claim follows.
\end{proof}

\subsection{Explicit formulas for configurations at time $t$}

The next lemma makes the connection 
between the CA evolution and the integer rings apparent. 

\begin{lem}\label{lemma: explicit evolution}
For $\sigma = 2$,  in $\mathbb{Z}_n[x]/(x^2 - 1)$,
\begin{equation}\label{expl2}
\begin{aligned}
(a + bx)^t 
& = \frac{1}{2}\left[(a + b)^t + (a - b)^t\right]\\
& + \frac{1}{2}\left[(a + b)^t - (a - b)^t\right]x.
\end{aligned}
\end{equation}
For $\sigma = 3$,  in $\mathbb{Z}_n[x]/(x^3 - 1)$,
\begin{equation}\label{expl3}
\begin{aligned}
(a + bx)^t 
& = \frac{1}{3} \left[(a + b)^t + (a + b\omega)^t + (a + b\omega^2)^t\right]\\
& + \frac{1}{3} \left[(a + b)^t + \omega^2(a + b\omega)^t + \omega(a + b\omega^2)^t\right]x\\
& + \frac{1}{3} \left[(a + b)^t + \omega(a + b\omega)^t + \omega^2(a + b\omega^2)^t\right]x^2.
\end{aligned}
\end{equation}
For $\sigma = 4$,  in $\mathbb{Z}_n[x]/(x^4 - 1)$,
\begin{equation}\label{expl4}
\begin{aligned}
(a + bx)^t
& = \frac{1}{4} \left[(a + b)^t + (a - b)^t + (a + bi)^t + (a - bi)^t\right]\\
& + \frac{1}{4} \left[(a + b)^t - (a - b)^t + i(a + bi)^t - i(a - bi)^t\right]x\\
& + \frac{1}{4} \left[(a + b)^t + (a - b)^t - (a + bi)^t - (a - bi)^t\right]x^2\\
& + \frac{1}{4} \left[(a + b)^t - (a - b)^t - i(a + bi)^t + i(a - bi)^t\right]x^3. 
\end{aligned}
\end{equation}
For $\sigma = 6$,  in $\mathbb{Z}_n[x]/(x^6 - 1)$,
\begin{equation}\label{expl6}
\begin{aligned}
(a+bx)^t
& = \frac{1}{6}\left[(a + b)^t + (a - b)^t + (a + b\omega)^t + (a + b\omega^2)^t +(a - b\omega)^t + (a - b\omega^2)^t
\right]\\
& + \frac{1}{6}\left[(a + b)^t - (a - b)^t + \omega^2(a + b\omega)^t + \omega(a + b\omega^2)^t - \omega^2(a - b\omega)^t - \omega(a - b\omega^2)^t
\right]x\\
& + \frac{1}{6}\left[(a + b)^t + (a - b)^t + \omega(a + b\omega)^t + \omega^2(a + b\omega^2)^t + \omega(a - b\omega)^t + \omega^2(a - b\omega^2)^t \right]x^2\\
& + \frac{1}{6}\left[(a + b)^t - (a - b)^t + (a + b\omega)^t + (a + b\omega^2)^t - (a - b\omega)^t - (a - b\omega^2)^t\right]x^3\\
& + \frac{1}{6}\left[(a + b)^t + (a - b)^t + \omega^2(a + b\omega)^t + \omega(a + b\omega^2)^t +\omega^2(a - b\omega)^t + \omega(a - b\omega^2)^t\right]x^4\\
& + \frac{1}{6}\left[(a + b)^t - (a - b)^t + \omega(a + b\omega)^t + \omega^2(a + b\omega^2)^t - \omega(a - b\omega)^t - \omega^2(a - b\omega^2)^t \right]x^5.
\end{aligned}
\end{equation}
\end{lem}

To clarify, say, the formula for $\sigma=6$, the expression 
in each square bracket is evaluated in $\Z[\omega]$ 
first (without 
the reduction modulo $n$), then the result, which must be in $6\Z$, is divided by $6$, and finally is reduced modulo $n$.

\begin{proof}
This follows from diagonalization of circulant matrices;
see, for example, \cite{davis1970circulant}.
\end{proof}

\subsection{The upper bounds}
In this subsection we prove the upper bounds in Theorem \ref{theorem: main of additive rule}.

\begin{lem} \label{lemma: upper bound}
For $n\ge 2$, $\pi_\sigma(n) \le \lambda_\sigma(\sigma n)$ for $\sigma = 2, 3$ and $\pi_6(n) \le \lambda_3(6n)$. 
Moreover, for $n\ge 3$,  $\pi_4(n)\le \lambda_4(n)$.
\end{lem}

\begin{proof}
We will show that, in all cases, $\Pi_\sigma(a, b; n)$ divides
the corresponding upper bound for all $a, b\in\Z_n$. 
Assume that $p\nmid\sigma$, which automatically holds when $p\ge 5$. In this case, we claim that
\begin{equation} \label{equation: max order divides exponent}
\Pi_\sigma(a, b; p^m) \bigm| \lambda_\sigma(p^m),
\end{equation}
which is clearly enough.
By Propositions~\ref{proposition: MOW 4.5} 
and~\ref{proposition: MOW 4.1}, $\Pi_\sigma(a, b; p^m) \bigm | p^{m - 1}(p^{\ord_\sigma(p)}-1)$. As $\ord_2(p)=1$,
$\ord_3(p) = 1$ when $p \mod 3 = 1$ and $\ord_3(p) = 2$ when $p \mod 3 = 2$,
and $\ord_4(p)=1$ when $p\mod 4=1$ and $\ord_4(p)=2$ when 
$p\mod 4=3$, Lemmas~\ref{lemma: lambda_2}--\ref{lemma: lambda_4}
imply (\ref{equation: max order divides exponent}). 

We 
now consider each 
$\sigma$ separately. Write $n=2^{m_2}3^{m_3}\cdots p^{m_p}$. 

We begin with $\sigma = 2$.
Note that (\ref{equation: max order divides exponent}) holds for $p = 3$, and we next consider powers of $2$. 
For $m = 1$ and $m=2$, it can be directly verified that 
$\Pi_2(a, b; 2^m) \bigm| 2$.
For $m \ge 3$, by Proposition \ref{proposition: MOW 4.5}, $\Pi_2(a, b; 2^m) \bigm| 2^{m - 2} \Pi_2(a, b; 2^2)$, and then
$\Pi_2(a, b; 2^m) \bigm| 2^{m-1}$. Therefore
$$
\Pi_2(a, b; 2^m) \bigm | \lambda_2(2^{m+1}),
$$
which, together with (\ref{equation: max order divides exponent}) and Proposition~\ref{proposition: MOW 4.4}, implies that 
$$\Pi_2(a, b; n)\bigm|
\text{\rm lcm} (\lambda_2(2^{m_2+1}), \dots, \lambda_2(p^{m_p}))
= \lambda_2(2n),
$$
by Lemma~\ref{lemma: lcm of exponent}.

We continue with $\sigma = 3$. Now, (\ref{equation: max order divides exponent}) holds for $p=2$ and we need to consider powers of
$3$. A direct verification shows that $\Pi_3(a, b; 3) \bigm| 6$.
For $m \ge 2$, $\Pi_3(a, b; 3^m) \bigm| 3^{m - 1}\Pi_3(a, b; 3)$ 
and so $\Pi_3(a, b; 3^m)\bigm|2\cdot 3^m$. By
Lemma~\ref{lemma: lambda_3},
$$\Pi_3(a, b; 3^m)\bigm| \lambda_3(3^{m+1})$$
and again 
(\ref{equation: max order divides exponent}), Proposition~\ref{proposition: MOW 4.4}, and Lemma~\ref{lemma: lcm of exponent} imply that  $\Pi_3(a, b; 3^m)\bigm|\lambda_3(3n)$.

Next in line is $\sigma = 4$. 
This time, a direct verification (by computer) 
shows that $\Pi_4(a, b; 2)$, 
$\Pi_4(a, b; 2^2)$, and $\Pi_4(a, b; 2^3)$ all divide $4$.
For $m \ge 3$, we then have $\Pi_4(a, b; 2^m) \bigm| 2^{m - 3}\Pi_4(a, b; 2^3)$, thus $\Pi_4(a, b; 2^m)\bigm| 2^{m - 1}$.
Now, if $n = 2^{m_2} 3^{m_3} \dots p^{m_p}$ and $m_2 \ge 2$ or $m_2 = 0$, the result follows similarly as for $\sigma=2$ 
or $\sigma=3$. 
If $m_2 = 1$, 
$$\Pi_4(a, b; 2\cdot 3^{m_3}\dots p^{m_p}) \bigm| \text{lcm}(4, \lambda_4(3^{m_3}), \dots, \lambda_4(p^{m_p})).$$
But
\begin{align*}
\text{lcm}(4, \lambda_4(3^{m_3}), \dots, \lambda_4(p^{m_p})) &= \text{lcm}(2, \lambda_4(3^{m_3}), \dots, \lambda_4(p^{m_p}))\\& = \text{lcm}(\lambda_4(2), \lambda_4(3^{m_3}), \dots, \lambda_4(p^{m_p}))=\lambda_4(n),
\end{align*}
 as long as one 
of the exponents $m_3,\ldots, m_p$ is nonzero, i.e., 
when $n \ge 3$.
The desired divisibility therefore holds.

Finally, we deal with $\sigma = 6$.  This time, a similar argument shows that
$\Pi_6(a, b; 2^{m_2}) \bigm| 3\cdot 2^{m_2}$
and
$\Pi_6(a, b; 3^{m_3}) \bigm| 2\cdot 3^{m_3}$, for all $m_2, m_3\ge 1$. 
So, 
$\Pi_6(a, b; n)$ divides
$$\text{lcm}(3\cdot 2^{m_2}, 
2\cdot 3^{m_3}, \ldots, \lambda_3(p^{m_p}))=\text{lcm}(\lambda_3(2 \cdot 2^{m_2}), \lambda_3(3 \cdot 3^{m_3}), \dots, \lambda_3(p^{m_p})) = \lambda_3(6n).$$
The desired divisibility is thus established in all cases.
\end{proof}

\subsection{The lower bounds}
\begin{lem} \label{lemma: lower bound for composite n}
If $n$ has prime decomposition $n = p_1^{m_1} \dots p_k^{m_k}$, then, for any $\sigma$, 
\begin{equation}\label{eq-lower1}
\begin{aligned}
\text{\rm lcm}\left(\pi_\sigma(p_1^{m_1}), \dots, \pi_\sigma(p_k^{m_k})\right) \le\pi_\sigma(n).
\end{aligned}
\end{equation}
\end{lem}
\begin{proof}
We identify $\mathbb{Z}_n$ by
$$
\mathbb{Z}_n \cong \mathbb{Z}_{p_1^{m_1}} \times \dots \times \mathbb{Z}_{p_k^{m_k}}.
$$ 
For the CA rule in the $j$th coordinate, we find $a_j, b_j\in \mathbb{Z}_{p_j^{m_j}}$ such that 
$\Pi_\sigma(a_j, b_j; p_j^{m_j}) = \pi_\sigma(p_j^{m_j})$.
Then a configuration repeats if and only if all $k$ coordinates simultaneously repeat.
\end{proof}

As a consequence of Lemma~\ref{lemma: lower bound for composite n}, it suffices to consider the cases when $n=p^m$. 
In each case below, our strategy is to find an $a,b\in \Z_{p^m}$
for which the dynamics never reduces the spatial period and 
such that $\Pi_\sigma(a, b; p^m)$ equals the upper bound 
given by Lemma~\ref{lemma: upper bound}.

\begin{lem} \label{lemma: lambda_2(2n)}
For $\sigma = 2$, we have $\pi_2(p^m)=\lambda_2(2p^m)$.
\end{lem}

\begin{proof} 
We first prove that $a - b \in \mathbb{Z}_{p^m}^\times$ 
implies that the spatial period never reduces. Indeed, such a 
reduction means that the coefficients of $1$ and $x$ 
in (\ref{expl2}) agree at some time $t\ge 1$, and then their
difference $(a-b)^t$ must vanish in $\mathbb{Z}_{p^m}$, 
a contradiction. 

We now assume that $p\ge 3$. 
By definition of $\lambda_2$, we 
can select $a$ and $b$ such that $\Lambda_2(a, -b; p^m) = \lambda_2(p^m)$; in particular, $a - b \in \mathbb{Z}_{p^m}^\times$. Let $k = \Pi_2(a, -b; p^m)$. Then, for some $\ell\ge 0$, $(a - bx)^{k + \ell} = (a - bx)^\ell$ in $\Z_{p^m}[x]/(x^2-1)$.
If we replace $x$ by any number $c\in \Z_{p^m}$ that satisfies $c^2-1=0\mod p^m$, 
we get an equality in $\Z_{p^m}$, so we can substitute $x=1$
to get $(a - b)^{k + \ell} = (a - b)^\ell\mod p^m$. 
As $a-b$ is invertible in $\Z_{p^m}$, $(a - b)^k = 1\mod p^m$. We conclude that
$\lambda_2(p^m) \le \Pi_2(a, - b; p^m) \le \pi_2(p^m)$. As the 
spatial period does not reduce,
the desired conclusion follows from the equality
$\lambda_2(p^m) = \lambda_2(2p^m)$ and Lemma \ref{lemma: upper bound}.

Finally, we assume that $p = 2$. In this case, we need to 
prove that $\pi_2(2^m)=\lambda_2(2^{m+1})$.  A direct 
verification shows that $\pi_2(2)=\pi_2(4)=2$, so we may assume 
that $m\ge 3$, in which case 
$\lambda_2(2^{m+1})=2^{m-1}$.
Pick a $c\in \Z_{2^{m+1}}^\times$ whose order equals 
$\lambda_2(2^{m+1})$. This is an odd number. Let $b=(c-1)/2$ and $a=b+1$, so that $a+b=c$ and $a-b=1$. Clearly
$b\le 2^m-1$, but then also $a\le 2^m-1$, as 
otherwise $c=2^{m+1}-1$, which has order $2$. 
It then follows from 
(\ref{expl2}) that $(a+bx)^{2^{m-1}}=1$ in 
$\Z_{2^m}[x]/(x^2-1)$. Moreover, the coefficient 
of $x$ in $(a+bx)^{2^{m-2}}$ cannot vanish in $\Z_{2^m}$, 
as otherwise $c^{2^{m-2}}=1\mod 2^{m+1}$. It follows that 
$\Pi_2(a, b; 2^m) =2^{m-1}$.
\end{proof}

\begin{lem}\label{lemma: lambda_3(3n)}
For $\sigma = 3$, we have $\pi_3(p^m)=\lambda_3(3p^m)$.
\end{lem}

\begin{proof}
We first show that, provided $a+b\omega\in \Z_{p^m}[\omega]^\times$, spatial period does not reduce. Indeed, 
if the spatial period reduces to $1$ at time $t\ge 1$, then 
from (\ref{expl3}) 
$$
\frac{1}{3}
\begin{bmatrix}
B & A\\
A & B
\end{bmatrix}
\begin{bmatrix}
(a + b\omega)^t\\
(a - b\omega)^t
\end{bmatrix}
=
\begin{bmatrix}
0\\
0
\end{bmatrix}\text{ in }\Z_{p^m}[\omega], 
$$
where $A = 1 - \omega$ and $B = 1 - \omega^2$.
This implies that $(a + b\omega)^t = 0$ in $\Z_{p^m}[\omega]$,
a contradiction.

This time, we first assume that ${p \neq 3}$ and select $a$ and $b$ such that $\Lambda_3(a, b; p^m) = \lambda_3(p^m)$. 
Then, if $k = \Pi_3(a, b; p^m)$, we have $(a + bx)^{k + \ell} = (a + bx)^\ell$, in $\mathbb{Z}_{p^{m}}[x]/(x^3 - 1)$, for some $\ell$. As $\omega^3=1$, we may replace 
$x$ with $\omega$ to get $(a + b\omega)^k = 1$ 
in $\Z_{p^m}[\omega]$.
As a result, $\lambda_3(p^m) \le \Pi_3(a, b; p^m)$. 
As the spatial period does not reduce, 
the desired conclusion follows from $\lambda_3(p^m) = \lambda_3(3p^m)$ and Lemma \ref{lemma: upper bound}.

It remains to consider $p=3$.
By direct verification, $\pi_3(3) = 6$, and we assume 
$m\ge 2$ from now on. Select $a = b = 1$.  By Proposition \ref{proposition: MOW 4.5}, $\Pi_3(1, 1; 3^m) = 2\cdot 3^{m^\prime}$, for some $m^\prime \in[1, m]$.
Also,  $(1 + x)^{2\cdot 3^m} = 1$ in $\mathbb{Z}_{3^m}[x]/(x^3-1)$, which can be easily verified by (\ref{expl3}) using 
$(1+\omega)^2 =\omega$, $(1+\omega^2)^2 =\omega^2$, 
and the fact, easily verified by induction, 
that $2^{2\cdot 3^m}=1\mod 3^{m+1}$.
So, it suffices to show that $(1 + x)^{2\cdot 3^{m - 1}} \neq 1$ in $\mathbb{Z}_{3^m}[x]/(x^3-1)$, and for this we verify that the constant term in (\ref{expl3})
does not equal $1$, that is, 
$$
(1 + 1)^{2\cdot 3^{m - 1}} + (1 + \omega)^{2\cdot 3^{m - 1}} + (1 + \omega^2)^{2\cdot 3^{m - 1}}\neq 3\text{ in } \Z_{3^{m+1}}[\omega].
$$
  Indeed, in  $\Z_{3^{m+1}}[\omega]$,
$(1 + \omega)^{2\cdot 3^{m - 1}} = (1 + \omega^2)^{2\cdot 3^{m - 1}} = 1 $ and, again by induction, $2^{2\cdot 3^{m - 1}}= 3^m + 1 $.
\end{proof}

\begin{lem}\label{lemma: lambda_4(n)}
For $\sigma = 4$, we have $\pi_4(p^m)=\lambda_4(p^m)$.
 \end{lem}
\begin{proof}
For any $p$, select $a$ and $b$ such that $\Lambda_4(a, b; p^m) = \lambda_4(p^m)$.
Then if $k = \Pi_4(a, b; p^m)$, we have $(a + bx)^{k + \ell} = (a + bx)^\ell$, in $\mathbb{Z}_{p^{m}}[x]/(x^4 - 1)$, for some $\ell$.
Replacing $x$ with $i$, we have $(a + bi)^k = 1$ in  $\Z_{p^m}[i]$.
As a result, $\lambda_4(p^m) \le \Pi_4(a, b; p^m)$.
Thus we only need to verify that the spatial period does not reduce. 
If it does, then for some $t$, by (\ref{expl4}), 
$$
\frac{1}{2}
\begin{bmatrix}
1 & 1\\
i & -i
\end{bmatrix}
\begin{bmatrix}
(a + bi)^t\\
(a - bi)^t
\end{bmatrix}
=
\begin{bmatrix}
0\\
0
\end{bmatrix} \text{ in }\Z_{p^m}[i],
$$
implying that $(a + bi)^t = 0$ in $\Z_{p^m}[i]$, 
a contradiction with $a+bi\in \Z_{p^m}[i]^\times$.
\end{proof}
 
 
\begin{lem}\label{lemma-s6-nored}
Assume that $\sigma = 6$, $n=p^m$, and that 
one of these two conditions on $a$ and $b$ is 
satisfied: $p\ne 3$ and $a+b\omega$ is
invertible $\Z_{p^m}[\omega]$; or $p=3$, $m\ge 2$, $a=1$ and $b=2$. Then the spatial period of $(a+bx)^t$ 
is $6$ for all $t\ge 0$. \end{lem}
\begin{proof}
If the period reduces to $2$, then
by (\ref{expl6}),  
$$
\frac{1}{6}
\begin{bmatrix}
A & B & A & B\\
B & A & B & A\\
-B & -A & B & A\\
A & B & -A & -B\\
\end{bmatrix}
\begin{bmatrix}
(a + b\omega)^t \\
(a + b\omega^2)^t \\
(a - b\omega)^t \\
(a - b\omega^2)^t 
\end{bmatrix}
= 
\begin{bmatrix}
0\\
0\\
0\\
0
\end{bmatrix} \text{ in }\Z_{p^m}[\omega],
$$
where $A = 1 - \omega$ and $B = 1 - \omega^2$.
Multiply rows, in order, by  $A$, $-B$, $B$,
$A$
and add. Using $B^2-A^2=3(2 \omega +1)$, we get that 
$(1+2 \omega)(a+b\omega)^t=0$ in $\Z_{p^m}[\omega]$. 
Multiplying instead by  $A$, $-B$, $-B$,
$-A$ gives $(1+2 \omega)(a-b\omega)^t=0$ in $\Z_{p^m}[\omega]$.
If $p\ne 3$, then $1+2 \omega \in\Z_{p^m}[\omega]^\times$ 
and so $(a+b\omega)^t=0$, a contradiction.
Assume now that $p=3$. Then we use the fact that 
Eisenstein norm $|1-2\omega|=7$, and so the norm of the product $|(1+2 \omega)(1-2 \omega)^t|=3\cdot 7^t$, 
which is not divisible by $3^m$ if $m\ge 2$, and so
$(1+2 \omega)(1-2 \omega)^t$ is nonzero in $\Z_{3^m}[\omega]$.

We next show that the spatial period does not reduce to $3$. 
 If it does, then by (\ref{expl6}),
$$
\frac{1}{3}
\begin{bmatrix}
1 & 1 & 1\\
1 & \omega^2 & \omega\\
1 & \omega & \omega^2
\end{bmatrix}
\begin{bmatrix}
(a - b)^t \\
(a - b\omega)^t \\
(a - b\omega^2)^t \\
\end{bmatrix}
= 
\begin{bmatrix}
0\\
0\\
0
\end{bmatrix} \text{ in }\Z_{p^m}[\omega].
$$
From this, we get that 
\begin{equation}\label{pi6-eq1}
(a - b)^t=(a - b\omega)^t =(a - b\omega^2)^t=0
\text{ in }\Z_{p^m}[\omega].
\end{equation}
Assume $p\ne 3$ first. Then, (\ref{pi6-eq1}) implies that
neither $a-b$ nor $a-b\omega$ is invertible in 
$\Z_{p^m}[\omega]$, and thus 
$p$ must divide $a-b$ and the norm
$a^2+b^2+ab$. Then $3ab=(a^2+b^2+ab)-(a-b)^2$ is also divisible by $p$,
and then so is $ab$. This implies that
$p\bigm|(a^2+b^2-ab)$,
and so $a+b\omega$ is not invertible, a contradiction.
If $p=3$, then (\ref{pi6-eq1}) is not satisfied for
$a=1, b=2$, as $(a-b)^t$ cannot vanish.
\end{proof}

\begin{lem}\label{lemma: lambda_3(6n)}
For $\sigma = 6$, we have $ \pi_6(p^m)=\lambda_3(6p^m)$.
\end{lem}
\begin{proof}
Assume first that $p \ge 5.$ Select any $a$ and $b$ such that $\Lambda_3(a, b; p^m) = \lambda_3(p^m)=\lambda_3(6p^m)$.
Then, if $k = \Pi_6(a, b; p^m)$, we have $(a + bx)^{k + \ell} = (a + bx)^\ell$, in $\mathbb{Z}_{p^{m}}[x]/(x^6 - 1)$, for some $\ell$.
Replacing $x$ with $\omega$, we have $(a + b\omega)^k = 1$ thus $\lambda_3(p^m) \le \Pi_6(a, b; p^m)$.

Next in line is ${p = 2}$. The claim is that $\pi_6(2^m) = 3 \cdot 2^m$. We may assume that $m\ge 3$, after a direct
verification for $m=1,2$. 
By Theorem~\ref{proposition: MOW 4.5}, $\Pi_6(1, 1; 2^m) = 3\cdot 2^{m^\prime}$, for some $m^\prime \in [1, m]$. 
Therefore, it suffices to show that there are infinitely many 
$\ell$ for which the equality
$$
(1 + x)^{3 \cdot 2^{m - 1} + \ell} = (1 + x)^{\ell}, \text{ in }\mathbb{Z}_{2^m}[x]/(x^6 - 1),
$$
is {\it not\/} satisfied. 
A necessary condition for this equality is that the constant 
terms in (\ref{expl6}) for both sides agree, which yields
\begin{align*}
\frac{1}{6}\bigg[
&2^{\ell}\left(2^{3\cdot 2^{m - 1}} - 1\right) + 
(1 + \omega)^{\ell}\left((1 + \omega)^{3\cdot 2^{m - 1}} - 1\right) + 
(1 + \omega^2)^{\ell}\left((1 + \omega^2)^{3\cdot 2^{m - 1}} - 1\right) + \\
& (1 - \omega)^{\ell}\left((1 - \omega)^{3\cdot 2^{m - 1}} - 1\right) + 
(1 - \omega^2)^{\ell}\left((1 - \omega^2)^{3\cdot 2^{m - 1}} - 1\right)
\bigg] =  0 \mod 2^m.
\end{align*}
As $1+\omega=-\omega^2$, $1+\omega^2=-\omega$, the second and 
third term vanish. The first term vanishes for large enough $\ell$. Moreover, as $(1-\omega)^2=-3\omega$ and  
$(1-\omega^2)^2=-3\omega^2$, 
$$(1-\omega)^{3\cdot 2^{m-1}}=(1-\omega^2)^{3\cdot 2^{m-1}}=
3^{3\cdot 2^{m-2}},$$ 
for $m\ge 3$. We obtain the necessary condition
\begin{equation}\label{s6-eq1}
(1 - \omega)^{\ell}\left[1 + (1 + \omega)^{\ell}\right]\left(3^{3\cdot 2^{m-2}} - 1\right) = 0 \mod 3\cdot 2^{m + 1}.
\end{equation}
If $\ell=1\mod 12$, then $(1-\omega)^\ell$ is 
a power of $3$ times $(1-\omega)$ and 
$(1+\omega)^\ell=-\omega^2$. 
By a simple induction argument, $3^{3\cdot 2^{m-2}}-1 = 2^m \mod 2^{m + 1}$. Then, if $\ell=1\mod 12$, (\ref{s6-eq1})
reduces to $3^{\ell'}\cdot 2^m=0\mod 3\cdot 2^{m+1}$, for some $\ell'\ge 1$,  
which is clearly false. This completes the proof for $p=2$.

Finally, we deal with $p=3$. 
We aim to prove $\pi_6(3^m) = 2 \cdot 3^m$, and we will
accomplish this by establishing the claim that 
$\Pi(1, 2; 3^m) = 2 \cdot 3^m$. We may, again, 
assume $m\ge 3$. Similarly to the previous case, 
it suffices to show that 
\begin{equation}\label{s6-eq15}
(1 + 2x)^{2 \cdot 3^{m - 1} + \ell} = (1 + 2x)^{\ell}, \text{ in }\mathbb{Z}_{3^m}[x]/(x^6 - 1),
\end{equation}
fails to hold
for infinitely many $\ell$, and we will assume that $\ell$ is 
large enough and $18 \bigm| \ell$.
As before, we show the constant terms in (\ref{expl6}) 
do not match. If they do, this expression needs to vanish 
modulo $2\cdot 3^{m+1}$:
\begin{equation}\label{s6-eq2}
\begin{aligned}
&(1 + 2)^\ell\left[(1 + 2)^{2\cdot 3^{m - 1}} - 1 \right]
+ (1 - 2)^\ell\left[(1 - 2)^{2\cdot 3^{m - 1}} - 1\right]\\
&+ (1 + 2\omega)^\ell \left[ (1 + 2\omega)^{2 \cdot 3^{m - 1}} - 1\right]
 + (1 + 2\omega^2)^\ell \left[ (1 + 2\omega^2)^{2 \cdot 3^{m - 1}} - 1\right]\\
&+ (1 - 2\omega)^\ell \left[ (1 - 2\omega)^{2 \cdot 3^{m - 1}} - 1\right]
+ (1 - 2\omega^2)^\ell \left[ (1 - 2\omega^2)^{2 \cdot 3^{m - 1}} - 1\right].
\end{aligned}
\end{equation}
As $(1 + 2\omega)^2 = (1 + 2\omega^2)^2=-3$, the first 
four terms all vanish when $\ell$ is large enough.  
For the fifth and sixth term, we first observe that
\begin{equation}\label{s6-eq3}
\begin{aligned}
(1 - 2\omega)^\ell & = \left[(1 - \omega) - \omega\right]^\ell 
= (- \omega)^\ell + \sum_{j = 1}^\ell {\ell \choose j}(1 - \omega)^j(-\omega)^{\ell - j} 
= 1 \text{ in }\Z_9[\omega].
\end{aligned}
\end{equation}
By a similar calculation, $(1 - 2\omega^2)^\ell = 1$ in $\Z_9[\omega]$.
Next, we have
\begin{equation}\label{s6-eq4}
\begin{aligned}
(1 - 2\omega)^{2 \cdot 3^{m - 1}} - 1 & = \left[(1 - \omega) - \omega \right]^{2 \cdot 3^{m - 1}} - 1 \\
& = -1+(- \omega)^{2 \cdot 3^{m - 1}} + 2 \cdot 3^{m - 1}(1 - \omega)(- \omega)^{2 \cdot 3^{m - 1} - 1} \\ 
& \quad + \frac{2 \cdot 3^{m - 1}(2 \cdot 3^{m - 1} - 1)}{2} ( 1 - \omega)^2 (- \omega)^{2 \cdot 3^{m - 1} - 2} \\
& \quad + \frac{2 \cdot 3^{m - 1}(2 \cdot 3^{m - 1} - 1)(2 \cdot 3^{m - 1} - 2)}{2 \cdot 3} ( 1 - \omega)^3 (- \omega)^{2 \cdot 3^{m - 1} - 3} \\
& \quad + \sum_{j = 4}^{2 \cdot 3^{m - 1}}{2 \cdot 3^{m - 1}\choose j}(1 - \omega)^j(- \omega)^{2\cdot 3^{m - 1} - j}\\
& = 2 \cdot 3^{m - 1}(1 - \omega)(- \omega^2) - 3^m(2 \cdot 3^{m - 1} - 1)\omega^2 \\ 
& \quad + 3^{m - 1}(2 \cdot 3^{m - 1} - 1)(2 \cdot 3^{m - 1} - 2)\omega(1 - \omega) \text{ in }\Z_{3^{m+1}}[\omega].
\end{aligned}
\end{equation}
Similarly, 
\begin{equation}\label{s6-eq5}
\begin{aligned}
(1 - 2\omega^2)^{2 \cdot 3^{m - 1}} - 1 
& = 2 \cdot 3^{m - 1}(1 - \omega^2)(- \omega) - 3^m(2 \cdot 3^{m - 1} - 1)\omega \\ 
& \quad + 3^{m - 1}(2 \cdot 3^{m - 1} - 1)(2 \cdot 3^{m - 1} - 2)\omega(\omega - 1) \text{ in }\Z_{3^{m+1}}[\omega].
\end{aligned}
\end{equation}
Combining (\ref{s6-eq3})--(\ref{s6-eq5}), we conclude that the expression (\ref{s6-eq2}) equals $3^m \mod 3^{m + 1}$. (We need 
$m\ge 3$ to ensure $3^{m+1}\bigm|3^{m-1}\cdot 3^{m-1}$, so that we
can ignore products of powers of $3$.) Therefore (\ref{s6-eq15})
does not hold, which concludes the proof for $p=3$. 

We also need that the spatial period is not reduced
in considered cases, which are all covered by Lemma~\ref{lemma-s6-nored}. 
\end{proof}


\begin{proof}[Proof of Theorem \ref{theorem: main of additive rule}]
The desired claims are established by Lemmas~\ref{lemma: upper bound}--\ref{lemma: lambda_4(n)}, and Lemma~\ref{lemma: lambda_3(6n)}.
\end{proof}

\section{PS with long temporal periods in non-additive rules} \label{section: general}

In this section, we prove Theorem~\ref{theorem: main of general rule}, by two explicit constructions. Our first rule resembles a car odometer, and is similar to others that  have previously appeared in the literature, see \cite{coven2009embedding}. We view this as  the most natural design, which 
also gives explicit constants $C(\sigma)$ and $N(\sigma)$, 
although the second construction based on prime partition is much shorter. 


\subsection{The odometer rule}
For a fixed integer $k\ge 2$, we define the state space 
$$\mathcal{S} = \mathbb{Z}_k \times \{\leftarrow, \bdot\} \times \{\ast, \bdot \} \times \{E, \bdot\},$$
which has cardinality $2^3k$. 
We call these four coordinates the \textbf{number}, \textbf{particle}, \textbf{asterisk}, and \textbf{end} coordinate, respectively. In words, each of the symbols $\leftarrow$, $\ast$, and 
$E$ can be present at a site 
in addition to a number, and $\bdot$ 
signifies its absence. 
We use abbreviations such as $(5, \leftarrow, \ast, E)=\overleftarrow{_E5^\ast}$, $(5, \leftarrow, \bdot, \bdot)=\overleftarrow{5}$, and  $(5, \bdot, \bdot, \bdot)={5}$.
To be consistent with the car odometer interpretation, we construct a \textit{right-sided} rule.
That is
$$
\xi_{t+1}(x) = f(\xi_t(x), \xi_t(x + 1)),
$$
or $\underline{\xi_t(x)}\xi_t(x + 1) \mapsto \xi_{t+1}(x)$.
Clearly, such a rule may be transformed to our
standard left-sided one by a vertical reflection.  

The rule is described in the following $14$ assignments, in which $I, J$ represent numbers in $\mathbb{Z}_k$ and addition is modulo $k$, $i, j$ 
represent elements in $\mathbb{Z}_k \setminus \{k - 1\}$, 
and $\arb$  stands for any state in $\mathcal{S}$:
\begin{multicols}{3}
\begin{enumerate}
\item  $\underline{I}\overleftarrow{i^\ast} \mapsto \overleftarrow{I}$; 
\item  $\underline{I}\overleftarrow{J} \mapsto \overleftarrow{I}$;
\item $\underline{I}\overleftarrow{(k-1)^\ast} \mapsto \overleftarrow{I^\ast}$;
\item  $\underline{\overleftarrow{I^\ast}} \arb \mapsto (I+1)$;
\item   $\underline{\overleftarrow{I}} \arb \mapsto I$;
\item   $\underline{I} \overleftarrow{_E(k - 1)} \mapsto \overleftarrow{I^\ast}$;
\item   $\underline{\overleftarrow{_Ei}} \arb \mapsto \overleftarrow{_E(i + 1)}$;
\item   $\underline{\overleftarrow{_E(k - 1)}} \arb \mapsto _E0$;
\item   $\underline{_EI} \overleftarrow{J^\ast} \mapsto \overleftarrow{_E0}$;
\item   $\underline{_EI} \overleftarrow{J} \mapsto \overleftarrow{_E0}$;
\item   $\underline{I}J \mapsto I$;
\item   $\underline{_EI}J \mapsto _EI$;
\item   $\underline{I}_EJ \mapsto I$;
\item   $\underline{I}\overleftarrow{_Ej} \mapsto I$.
\end{enumerate}
\vspace*{\fill}
\end{multicols}
In all cases not covered above, the 
rule leaves the current state unchanged: $\underline{c_0}c_1 \mapsto c_0$. We view the rule on $[0,\sigma-1]$ with 
periodic boundary, that is, within one spatial period
of the PS.  

Our construction simulates the dynamics of an odometer on the number coordinate. The three auxiliary coordinates are needed
for the update rule to be a CA. We now give a less 
formal description.
The end position indicator $E$ marks the right end of our 
interval with periodic boundary. Hence, there has to be 
exactly one $E$ and it is designed so that it does not appear or disappear (see assignments 7--10 and 12--14).
The $\leftarrow$ is a left-moving particle
(assignments 1--10), 
marking the site on which the number coordinate may 
add 1 in the next step.
The number marked by an $E$ adds 1 if its site also contains 
a particle, i.e., its  particle coordinate is an $\leftarrow$ 
 (assignments 7 and 8), and updates to 0 when 
 an $\leftarrow$ is to its right (assignments 9 and 10).  
The number coordinates not marked by an $E$ 
add 1 if and only if the asterisk coordinate is $\ast$ (see assignment 4 and 5).
The symbol $\ast$ plays the role of carry in addition and can appear and disappear:
it appears if the $E$ position has number $k - 1$, then it moves along with the particle (see assignment 6)
if its number coordinate is $k - 1$ (see assignment 3), 
and disappears if there is no carry (see 1) or if it arrives to the $E$ position (see 9).

\begin{table}[ht!]
\caption{An odometer PS for $\sigma=3$, $k=10$.}\label{odo-ex}
\begin{center}
\begin{tabular}{lllll}
 0 & 0 & $\overleftarrow{_E0}$ & \\
 0 & 0 & $\overleftarrow{_E1}$ & (11, 14, 7)\\
   & \vdots  & \\
 0 & 0 & $\overleftarrow{_E9}$ & (11, 14, 7)\\
 0 & $\overleftarrow{0^\ast}$ & $_E0$ & (11, 6, 8)\\
 $\overleftarrow{0^{\phantom *}}$ & 1 & $_E0$ & (1, 4, 12)\\
 0 & 1 & $\overleftarrow{_E0}$ & (5, 13, 10)\\
 0 & 1 & $\overleftarrow{_E1}$ & (11, 14, 7)\\
   & \vdots  & \\
 0 & 9 & $\overleftarrow{_E9}$ & (11, 14, 7)\\
 0 & $\overleftarrow{9^\ast}$ & $_E0$ & (11, 6, 8)\\
 $\overleftarrow{0^\ast}$ & 0 & $_E0$ & (3, 4, 12)\\
 1 & 0 & $\overleftarrow{_E0}$ & (4, 13, 9)\\
   & \vdots  & \\
   9 & 9 & $\overleftarrow{_E9}$ & (11, 14, 7)\\
   9 & $\overleftarrow{9^\ast}$ & $_E0$ & (11, 6, 8)\\
   $\overleftarrow{9^\ast}$ & 0 & $_E0$ & (3, 4, 12)\\
   0 & 0 & $\overleftarrow{_E0}$ & (4, 13, 9). \\
\end{tabular}
\end{center}
\end{table}

Any rule with the above fourteen \textbf{odometer assignments}  is called an \textbf{odometer CA} and generates a PS of temporal period at least $k^\sigma$ , called \textbf{odometer PS}. 
This shows that  
$\max_fX_{\sigma, 8k}(f) \ge k^\sigma$.
To give an example, let $L = 00\dots \overleftarrow{_E0}$ be the configuration consisting of $(\sigma - 1)$ 0's and a $\overleftarrow{_E0}$. 
When $\sigma = 3$, $k = 10$, then the PS is given in Table~\ref{odo-ex}, where the relevant assignments are given in the parentheses.  The PS has temporal period 
$1199>10^3=k^\sigma$. We summarize the result of this 
section, which provides the best lower bound we have on $\max_fX_{\sigma, n}(f)$.

\begin{prop} \label{prop:odo} 
There exists a CA rule $f$ so that  
$X_{\sigma,n}(f)\ge \lfloor n/8\rfloor^\sigma$.
\end{prop}
 
 The shortcoming of this construction is that 
 it does not ensure that $Y_{\sigma, n}(f) = \Theta(n^\sigma)$, as the odometer rule, as it stands, has other PS with much shorter temporal periods. For example, in the CA from 
 Table~\ref{odo-ex}, the configuration $123$ 
 is fixed due to the assignment 11, and so it 
 generates a PS with temporal period 1. We provide 
 the remedy in the next subsection.

\subsection{The odometer rule with automata}\label{section:odometer}	
	To prevent short temporal periods, we need to extend the state space.
	The strategy is to introduce a second layer to each state, which encodes two finite automata that 
determine whether a configuration is legitimate, i.e., either itself or one of its updates is included in the above odometer PS. A
legitimate configuration will generate the PS with long temporal period, while an illegitimate one will eventually end 
up in a spatially constant configuration. 
	
	\begin{defi}
	Consider the state space $\mathbb{Z}_k \times \{\leftarrow, \bdot\} \times \{\ast, \bdot \} \times \{E, \bdot\} \times \mathcal{A}$  of the odometer CA, where $\mathcal{A}$ is any finite set. 
	A configuration on $[0,\sigma-1]$ 
	is \textbf{legitimate} if the following three conditions are satisfied: (1) there is exactly one site that contains an $\leftarrow$; (2) there is exactly one site that contains an $E$; (3) if a site contains $\ast$, then this site contains an $\leftarrow$ but does not contain an $E$.
	\end{defi}
	
	\begin{lem}\label{lemma-enter-PS}
	Any odometer rule starting from any legitimate  configuration eventually enters the odometer PS.
	\end{lem}
	\begin{proof}
	\textit {Case 1}. 
	An inductive argument shows that any legitimate configuration in the form of $a_0\dots \overleftarrow{_Ea_{\sigma - 1}}$ generates the odometer PS.
	 
	\noindent\textit{Case 2}. 
	Suppose that a legitimate configuration does not contain an $\ast$ and thus is of the form $a_0\dots \overleftarrow{a_j}\dots _Ea_{\sigma - 1}$. 
	Then by  assignments 2 and 5, the $\leftarrow$ moves left until $\overleftarrow{a_0}\dots _Ea_{\sigma - 1}$ and then updates to $a_0\dots \overleftarrow{_E0}$ because of assignments 5 and 10, reducing to Case 1.
	
\noindent	\textit{Case 3}. 
	A legitimate configuration $a_0\dots \overleftarrow{a_j^\ast}\dots _Ea_{\sigma - 1}$, $a_j<k-1$,  updates to $a_0\dots \overleftarrow{a_{j-1}}(a_{j} + 1)\dots _Ea_{\sigma - 1}$ because of assignments 1 and 4, or to $a_0\dots \overleftarrow{_Ea_{\sigma - 1}}$, reducing to either Case $2$ or Case 1. 
	
\noindent	\textit{Case 4}. 
	A legitimate configuration 
	$a_0\dots \overleftarrow{(k-1)^\ast}\dots _Ea_{\sigma-1}$
	(with the $\leftarrow$ at position $j$) becomes $a_0\dots \overleftarrow{a_{j - 1}^*}0\dots _Ea_{\sigma - 1}$, which is reduced to Case 3 when $a_{j - 1} < k - 1$ . If $a_{j - 1} = k - 1$, repeated updates eventually reduce to Case 3 or Case 1.
	\end{proof}

We now define the augmented state space for our two-layer
construction of the \textbf{odometer rule with automata}: $$\mathcal{S}_A = \left(\mathbb{Z}_k \times \{\leftarrow, \bdot\} \times \{\ast, \bdot\} \times \{E, \bdot\}\times \mathcal{E} \times \mathcal{A}\right) \cup \{T\},$$
where 
$\mathcal{E} = \{(0, 0), (1, 0), \dots, (\sigma - 1, 0), (1, 1), (2, 1),\dots, (\sigma - 1, 1), T_1\}$ comprises states of a finite automaton, called \texttt{END-READER}; 
$\mathcal{A} = \{0, 1, \dots, \sigma, T_2 \}$ comprises states of another finite automaton, called \texttt{ARROW-READER}; and 
$T$ is the special terminator state that erases the configuartion once it appears.
We regard the first four components --- those from the odometer rule above --- as the first layer of a state, and the two automata components as the second layer.
	
We proceed to specify the rule. 
The first layer updates according to the previous odometer assignments. In addition, we include the assignment
	\begin{itemize}
	\item
	$\underline{(I, \bdot, *, \bdot)} s \mapsto T$ and $\underline{(I, \bdot, *, E)} s \mapsto T$ for all $s\in \mathcal{S}_A$.
	\end{itemize}
That is, if the first layer of a state contains an $\ast$ but not an $\leftarrow$, the state updates to $T$. Such an update 
will happen in any configuration that is illegitimate  due to having an $\ast$ but not an $\leftarrow$.
	
The next assignment spells out
the role of $T_1$, $T_2$, and $T$: 
\begin{itemize}
\item For any site $x$, if either $x$ or $x+1$ is in the state $T$ or at least one of the second layers of $x$, $x+1$ contains a $T_1$ or a $T_2$, then $x$ updates its state to $T$. 
\end{itemize}	
A configuration that contains a $T_1$, a $T_2$ or a $T$ is called \textbf{terminated}. Any terminated configuration will eventually update to the constant configuration consisting 
of all $T$'s, thus reduce the spatial period to $1$.

The transition function  $\delta_E$ of the finite automaton \texttt{END-READER} $ = (\mathcal{E}, \{E, \bdot\}, \delta_E, (i,j), T_1)$ reads the end coordinate and is given in Fig. \ref{fig:Transition function of E-reader}; its initial state $(i, j)$ can be any state in $\mathcal{E}$. 
From time $t$ to time $t+1$, an \texttt{END-READER} at position $x$ reads the state on its first layer, updates its state according to $\delta_E$, then ``moves'' to $x - 1$. 
This left shift of the entire \texttt{END-READER} configuration is allowed 
as we are constructing a right-sided rule.
According to the odometer assignments, the $E$ position in a configuration does not appear or disappear  and does not move. 
As a result, the \texttt{END-READER} counts the number of $E$'s.
\begin{lem}\label{lem: END-READER}
Every configuration with 0 or at least 2 sites containing an $E$ will be terminated for any initial state of the \texttt{\tt END-READER}. 
Conversely, starting from a configuration whose first layer is $= 00\dots \overleftarrow{_E0}$, no \texttt{\tt END-READER} ever reaches $T_1$ unless it starts there.  
\end{lem}
 
\begin{proof}
Start with a configuration with 0 or 2 more states that 
contain an $E$.
Suppose that it is never terminated by the \texttt{END-READER}. 
Then there is a time $t$ and a 
position $x$ such that the state of 
the \texttt{END-READER} is $(0, 0)$, as it is clear from Fig.~\ref{fig:Transition function of E-reader}. Within $\sigma$ time steps from $t$, 
the \texttt{END-READER} transitions to $T_1$.
The converse result is also clear from Fig.~\ref{fig:Transition function of E-reader}.
\end{proof}
	
	\begin{figure}	
	\begin{center}
	\begin{tikzcd}
		(0, 0) \arrow[r, "\bdot "] \arrow [dr, "E"] \arrow[rrrr, leftarrow, bend left = 30 , "E"] \arrow[drrrr, leftarrow, bend right = 80, "\bdot "]& (1, 0) \arrow[r, "\bdot "] \arrow [dr, "E"] & (2, 0) \arrow[r, "\bdot "] \arrow [dr, "E"] & \cdots \arrow[r, "\bdot "] \arrow [dr, "E"] & (\sigma - 1, 0) \arrow [dr, "\bdot "] \\
		& (1, 1) \arrow[r, "\bdot "] \arrow[rrrr, crossing over, bend right = 40 , "E"] & (2, 1) \arrow[r, "\bdot "] \arrow[rrr, crossing over, bend right = 30 , "E"] & \cdots \arrow[rr, crossing over, bend right = 15 , "E"] \arrow[r, "\bdot "] & (\sigma - 1, 1) \arrow[r, , "E"]  & T_1\\
  	\end{tikzcd}
  	\caption{The transition function $\delta_E$ for \texttt{END-READER}.} \label{fig:Transition function of E-reader}
  	\end{center}
  	\end{figure}

We also need to terminate illegitimate configurations  with
 0 or at least 2 arrows. First,  a configuration with 2 or more arrows can be handled by adding the following assignment:
 \begin{itemize}
 \item 
 $\underline{s_1} s_2 \mapsto T$, for all $s_1, s_2\in \mathcal{S}_A$ such that $s_1$, $s_2$ both contain an $\leftarrow$.
 \end{itemize}

\begin{lem}\label{lemma: two more arrows terminate}
Assume $k > \sigma$. Let $L$ be a configuration that is never terminated by the \texttt{\tt END-READER} and such that at least two states of $L$ contain an $\leftarrow$. Then $L$ will be eventually terminated. 
\end{lem}
\begin{proof}
Since $L$ is not terminated by the \texttt{END-READER}, there is exactly one state of $L$ that contains $E$.
Assume that the two states with $\leftarrow$ are not adjacent, as otherwise the configuration is terminated immediately.   
Note that the arrow at the $E$ position stays there for $k$ updates and other arrows move left at every update. 
As $k > \sigma$, two arrows will eventually be adjacent.
\end{proof}
	
	Due to Lemma~\ref{lemma: two more arrows terminate}, it suffices to enlist a finite automaton  
whose mission is to terminate configurations with no $\leftarrow$. This automaton is the \texttt{ARROW-READER} that reads the particle and end coordinates and is given by $(\mathcal{A}, \{\leftarrow, \bdot\} \times \{E, \bdot\}, \delta_A, (i,j), T_2)$, where the transition function $\delta_A$ is described in Fig. \ref{fig:Transition function of A-reader} and its initial state is any state in $\mathcal{A}$. 
	From time $t$ to time $t+1$, an  \texttt{ARROW-READER}
	at site $x$ updates its state according to $\delta_A$ and stays at the same position $x$. 
	According to the odometer assignments, an $\leftarrow$ must appear at the $E$ position within $\sigma$ updates if there is 
	at least one $\leftarrow$.
Hence, the \texttt{ARROW-READER} terminates a configuration that fails this condition. The effect of this 
	automaton is summarized in the following lemma.

	
	
	\begin{figure}	
	\begin{center}
  	
  	\begin{tikzcd}
		0 \arrow[r, "{(\bdot , E)}"] \arrow[loop above, "w" ] & 1 \arrow[r, "{(\bdot , E)}"] \arrow[l, bend left = 20, "w"]& 2 \arrow[r, "{(\bdot , E)}"] \arrow[ll, bend left = 35, "w "]& \cdots  \arrow[r, "{(\bdot , E)}"] & \sigma \arrow[r, "{(\bdot , E)}"] \arrow[llll, bend left = 35, "w"]& T_2\\		  	
	\end{tikzcd}
	\caption{The transition function $\delta_A$ of the \texttt{ARROW-READER}. Here $w$ is any symbol in $\{ \leftarrow, \bdot\} \times \{ E, \bdot \}\setminus \{(\bdot, E)\}$.} \label{fig:Transition function of A-reader}
  	\end{center}
  	\end{figure}
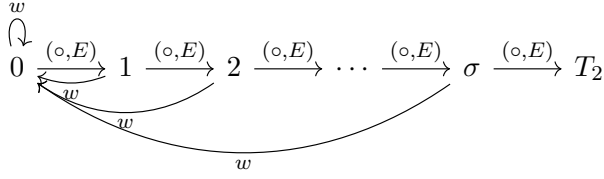
	
\begin{lem}\label{lem: ARROW-READER}
	Every configuration with no  $\leftarrow$ is eventually terminated for any initial state of the \texttt{\tt ARROW-READER}. Conversely, starting from a configuration whose first layer is $00\dots \overleftarrow{_E0}$, no \texttt{\tt ARROW-READER} ever reaches $T_2$ unless it starts there.
	\end{lem}
	

The next proposition provides our first proof of 
Theorem~\ref{theorem: main of general rule}.

\begin{prop} \label{prop:odoa} Let $S(\sigma)= 16\sigma(\sigma + 2)$.
For the rule $f$ defined in this subsection, 
we have $X_{\sigma,n}(f)=Y_{\sigma,n}(f)\ge \lfloor n/S(\sigma)\rfloor^\sigma$ for $n\ge (\sigma+2)S(\sigma)+1$. 
\end{prop}
	
\begin{proof}
Observe that $|\mathcal{S}_A|=S(\sigma)\cdot k+1$. 
 For a number of states $n$, let $k=\lfloor (n-1)/S(\sigma)\rfloor$. 
 Encode the odometer rule with automata on $S(\sigma)\cdot k+1$
 states, and make any leftover states immediately transition 
 to $T$. 
Let $L\in \mathcal{S}_A^\sigma$ be a configuration with its first layer is $00\dots\overleftarrow{_E0}$; on the second layer, the \texttt{END-READER}'s are at state $(0, 0)$ and the \texttt{ARROW-READER}'s are at state $0$. 
Then the configuration is not terminated by either \texttt{END-READER} or \texttt{ARROW-READER}, by 
Lemmas~\ref{lem: END-READER} and~\ref{lem: ARROW-READER}.
Then the global configuration restricted on the first layer is the one of odometer CA, which has temporal period 
at least $k^\sigma$. 
Therefore, $X_{\sigma, n}(f) \ge k^\sigma=\lfloor n/S(\sigma)\rfloor^\sigma$. 

Furthermore, note that any illegitimate configuration in $\mathcal{S}_A^\sigma$, as well as
any configuration not in  $\mathcal{S}_A^\sigma$, 
will eventually 
produce the constant configuration of all $T$s with spatial period $1$, by Lemmas~\ref{lem: END-READER}--\ref{lem: ARROW-READER}. Furthermore, 
any legitimate configuration on the first layer will eventually 
update to a configuration whose first layer is in the odometer PS (by Lemma~\ref{lemma-enter-PS}), and will never be terminated by the second layer 
that is not already in one of the terminator states
(by Lemmas~\ref{lem: END-READER} and~\ref{lem: ARROW-READER}).
Therefore, $Y_{\sigma, n}(f)=X_{\sigma, n}(f)$.
\end{proof}

\subsection{The prime partition rule}

We begin with a simple consequence of the prime number theorem.
	
\begin{lem} \label{lem: sigma primes}
For an arbitrary $\sigma > 0$, and for large enough $n$,   there are $\sigma$ primes $p_0, \dots, p_{\sigma- 1} \in [\frac{n-1}{2\sigma}, \frac{n-1}{\sigma}]$.
\end{lem}

Assume that $n$ is large enough so that Lemma \ref{lem: sigma primes}
holds. 
Find disjoint sets
$P_0, \dots, P_{\sigma - 1} \subset \mathbb{Z}_n\setminus\{0\}$
such that $|P_j| = p_j$, for $j = 0, \dots, \sigma - 1$.
This can be achieved since $p_0 + \dots + p_{\sigma - 1} \le n-1$. 
The state $0\in \mathbb{Z}_n \setminus (P_0 \cup \dots \cup P_{\sigma - 1})$
will play the role of the terminator.
Let $\phi_j: P_j \to P_j$ be a cyclic permutation of the 
$p_j$ states. Keeping the right-sided convention 
from the Section~\ref{section:odometer}, we define the CA rule $f$ as follows:
$$
f(s,s')=
\begin{cases}
 \phi_j(s) &\text{if } s\in P_j\text{ and }s' \in P_{(j + 1)\mod\sigma} \text{ for some } j \in \{0,\ldots, \sigma-1\}\\
 0 &\text{otherwise}
\end{cases}.
$$

\begin{prop}\label{prop:prp}
For $f$ defined above, we have $X_{\sigma,n}(f)=Y_{\sigma,n}(f)$
and $\liminf_{n\to\infty}n^{-\sigma}Y_{\sigma,n}(f)\ge
(2\sigma)^{-\sigma}$. 
\end{prop}
 
\begin{proof}
Call a configuration $s_{ 0} s_{1} \dots s_{\sigma - 1}$ {\it regular\/} if there exists an $\ell$
so that 
$s_j\in P_{(j+\ell)\mod \sigma}$, $j=0,\ldots,\sigma-1$. 
To show that $X_{\sigma, n}(f) \ge (n-1)^\sigma/(2\sigma)^\sigma$, run the rule starting from any regular configuration. Such a  configuration
appears again for the first time after  $p_0p_1\dots p_{\sigma - 1} \ge (n-1)^\sigma/(2\sigma)^\sigma$ updates.
To show that $Y_{\sigma, n}(f) = X_{\sigma, n}(f)$, observe 
that any non-regular initial configuration eventually 
ends up in the constant configuration of all $0$s.  
\end{proof}

\section{Discussion and open problems} \label{section: conclusion}

In this paper, we continue our study of the shortest and the longest temporal periods of a PS for a fixed spatial period $\sigma$.
While we are able to construct a rule whose longest temporal period grows as $n^\sigma$ for large $n$, more precise 
results remain elusive even for $\sigma=3$. We start our discussion with this case. 

We call an $n$-state rule that has a PS with spatial period $\sigma$ and temporal period $T(\sigma, n)$ as \textbf{maximum cycle length} (\textbf{MCL}) rule.
For $\sigma = 3$, our computations demonstrate that an MCL rule exists for $n \le 20$.
More precisely, the number of MCL rules is $1$ for $n=2$ (out of $2^4$ rules), $12$ for $n=3$ (out of $3^9$ rules) and $732$ for $n=4$ (out of $4^{16}$ rules).
These numbers match the first three terms of the sequence
\begin{equation}\label{curious}
(-1)^{k}7^{2k}E_{2k}\left(\frac{3}{7}\right), k = 0, 1, 2, 3, \ldots = 1, 12, 732, 109332, \ldots,
\end{equation}
where $E_n$ are the Euler polynomials.
Unfortunately, it is hard to traverse all of the $5^{25} \approx 2.98 \times 10^{17}$ $5$-state rules to count the number of MCL ones, so we merely state an open question.

\begin{ques} Assume $\sigma=3$. 
Does there exist an MCL rule for any number of states $n\ge 2$? 
If so, is the number of MCL rules given by (\ref{curious}) 
for all $n$, 
or is the connection just a curious coincidence for $n\le 4$? 
\end{ques}

If $X_{\sigma, n}(f) = T(\sigma, n)$, then automatically $Y_{\sigma, n}(f) = X_{\sigma, n}(f) = T(\sigma, n)$, as the PS goes through all configurations with
number of states $n$ and spatial period $\sigma$.
However, for $\sigma\ge 4$, an MCL may not exist, as demonstrated
for $n=3$ by Table \ref{table: long cycle}, 
and therefore the maxima of $X_{\sigma, n}$ and $Y_{\sigma, n}$ may differ. This motivates our next question.


\begin{table}[ht!]
\centering
\caption {Maximal temporal period for $n=3$ and 
spatial periods $\sigma\le 10$. We also give 
$N_X$, and $N_Y$, the numbers of rules that realize 
the respective maxima.}   
\begin{tabular}{c|ccccc}
$\sigma$ & $\max_fX_{\sigma, 3}(f)$ & $N_X$   &$\max_fY_{\sigma, 3}(f)$ & $N_Y$ 		   & $T(\sigma, 3)$\\ \hline \hline
1 & 3   & 1458  & 3      & 1458 & 3 \\ 
2 & 6   & 216	& 6   & 216   & 6 \\ 
3 	& 24  & 12 & 24  & 12    & 24 \\ 
4    & 40  & 12   & 32  & 72  & 72\\ 
5 		& 120 & 2 & 120	& 2    & 240\\ 
6     & 111 & 6	& 84  & 42  & 696\\
7    	& 1967& 12  & 546	& 2  & 2184\\
8 	& 904 & 12   & 896 & 24 & 6480\\ 
9    	& 9207& 12  & 1809& 12  & 19656\\ 
10 	& 10490 & 6  & 410 & 12 & 58800\\ 
\end{tabular} 
\label{table: long cycle}
\end{table}

\begin{ques} What is the asymptotic behavior of $\max_f X_{\sigma,3}(f)$ as $\sigma$ grows? Or of $\max_f X_{\sigma,n}(f)$ for an arbitary fixed $n$? Making $n$ large first, 
what is the asymptotic behavior of 
$$
\liminf_{n\to\infty} n^{-\sigma}\max_f X_{\sigma,n}(f)
$$
for large $\sigma$? (See Proposition~\ref{prop:odo} for an exponentially small lower bound.) 
The same questions can be posed for $Y_{\sigma,n}$ 
(for which Propositions~\ref{prop:odoa} and~\ref{prop:prp} provide even smaller
lower bounds). 
\end{ques}

To discuss the relation between $X_{\sigma, n}$ and 
$Y_{\sigma, n}$ for additive rules, let 
$\rho_\sigma(n)=\max_{f \in A_n}Y_{\sigma, n}(f)$.
As it is clear from Table \ref{table: experimental results}, 
$\pi_\sigma(n)$ and $\rho_\sigma(n)$ may differ, even for $\sigma = 2$ or $3$. This suggests our next question.
 
\begin{ques} Fix a $\sigma\ge 2$. 
Is there an explicit formula for $\rho_\sigma(n)$, in terms of $n$, at least for small $\sigma$? Can one characterize $n$ 
for which $\pi_\sigma(n)=\rho_\sigma(n)$?
\end{ques}

\begin{table}[H]
\centering
\caption {Maximum of shortest and longest temporal periods of additive rules, for $\sigma = 2, 3$ and $n = 2, \dots, 20$} 
\begin{tabular}{c|ccccccccccccccccccc}
$n$                                         & 2 & 3 & 4 & 5  & 6 & 7 & 8  & 9  & 10 & 11  & 12 & 13 & 14 & 15 & 16 & 17  & 18 & 19 & 20 \\ \hline\hline
$\rho_2(n)$ & 2 & 2 & 2 & 4  & 2 & 6 & 2  & 2  & 4  & 10  & 2  & 12 & 6  & 4  & 2  & 16  & 2  & 18 & 4  \\
$\pi_2(n)$ & 2 & 2 & 2 & 4  & 2 & 6 & 4  & 6  & 4  & 10  & 2  & 12 & 6  & 4  & 8  & 16  & 6  & 18 & 4  \\ \hline
$\rho_3(n)$ & 3 & 6 & 3 & 24 & 6 & 6 & 3  & 6  & 24 & 120 & 6  & 12 & 6  & 24 & 3  & 288 & 6  & 18 & 24 \\
$\pi_3(n)$ & 3 & 6 & 6 & 24 & 6 & 6 & 12 & 18 & 24 & 120 & 6  & 12 & 6  & 24 & 24 & 288 & 18 & 18 & 24
\end{tabular}
\label{table: experimental results}
\end{table}\

For a prime power $p^m$, we define the function 
$\ub_\sigma(p^m)$ to be the upper bound obtained 
from Propositions~\ref{proposition: MOW 4.3},~\ref{proposition: MOW 4.1} and ~\ref{proposition: MOW 4.5}. That is, 
$\ub_1(p)=p-1$; $\ub_\sigma(p)=p^{\ord_\sigma(p)}-1$ if 
$p \nmid \sigma$ and $\sigma\ge 2$; $\ub_\sigma(p)=
p^k\cdot \ub_{\sigma/p^k}(p)$ if $k\ge 1$ is the largest power of
$p$ dividing $\sigma$; and 
$\ub_\sigma(p^m)=p\cdot\pi_\sigma(p^{m-1})$ if $m\ge 2$. 
It is common that $\pi_\sigma(p^m)=\ub_\sigma(p^m)$, 
most notably for $\sigma=5$.  

\begin{ques}\label{q-sigma5}
Is it true that, for all prime powers $p^m$, 
$\pi_5(p^m)=\ub_5(p^m)$?
\end{ques}

We have checked that there are no counterexamples 
to the ``yes'' answer on Question~\ref{q-sigma5}
for all $p^m$ such that $p\le 50$ and $\ub_5(p^m)\le 10^5$. 
As counterexamples should be harder to come by for 
larger $p$ (more $a$ and $b$ to choose from) and 
for larger $m$ (less chance for $\Pi(a,b;p^m)$ to 
be equal to $\Pi(a,b;p^{m-1})$), we conjecture that the answer to Question~\ref{q-sigma5} is indeed 
affirmative. We also remark, that, if this conjecture holds, 
there is an explicit formula for $\pi_5(n)$ 
for all $n$, due to Lemma~\ref{lemma: lower bound for composite n} and Proposition~\ref{proposition: MOW 4.4}. 

It is not {\it always\/} true that $\pi_\sigma(p^m)=\ub_\sigma(p^m)$.  Table~\ref{table-pi-ub} contains 
 a list of examples of inequality 
we have found for $\sigma\le 50$. One hint that the table 
offers is easy to prove and we do 
so in the next proposition.

\begin{table}[ht!]\centering
\caption{Examples with $\pi(p^m)<\ub(p^m)$.  An arrow indicates a range of powers.}\label{table-pi-ub}
    \begin{tabular}{ c|ccccccccccccccccccc}
    \hline
     $\sigma$ & 2 & 4 & 7  & 8 & 11 & 13 & 14 & 16  
    \\ \hline\hline
    $p^m$ & $2^2$ & $2^{2\rightarrow 3}$ & 3 & 
    $2^{2\rightarrow 4}$ & 2 & 2 & 3  & $2^{2\rightarrow 5}$
    \\ 
    $\pi_\sigma(p^m)$ &  2 & 4 & 364 & 8 & 341 & 819 &364 & 16  
    \\ 
    $\ub_\sigma(p^m)$ &  4 & 8 & 728 & 16 & 1023 & 4095 & 728 & 32  
 \\ \hline
    \end{tabular}
    
    \vspace{0.2cm}
    
    \begin{tabular}{ c|cccccccccccccc}
    \hline
     $\sigma$   & 21 & 22 
     & 26 & 32 & 42 & 44
    \\ \hline\hline
    $p^m$    
    & 3 & 2 & 2 & $ 2^{2\rightarrow 6}$ & 3 & 2
    \\ 
    $\pi_\sigma(p^m)$   & 1092 
    & 682 & 1638 & 32 & 1092 & 1364
    \\ 
    $\ub_\sigma(p^m)$     & 
    2184 
    & 2046 & 8190 & 64 & 2184 & 4092
 \\ \hline
    \end{tabular}
\end{table}


\begin{prop} \label{prop-powers2}
Assume that $\sigma=2^k$, $k\ge 1$. 
Then $\pi_\sigma(2^m)=2^k$ for all $m\le k+1$, 
but $\pi_\sigma(2^{k+2})=2^{k+1}$. 
\end{prop} 

\begin{proof}
When $n=2$, $(1+x)^{2^k}=1+x^{2^k}=0$ in 
$\Z_{2}[x]/(x^\sigma-1)$. This implies that, 
for any $m$,
when $a$ and $b$ are both odd, all states are eventually 
divisible by $2$, and then by additivity $(a+bx)^t=0$ 
for large enough $t$. Clearly the same is true 
when $a$ and $b$ are both even. If $a$ is odd and $b$ is even, 
$$
(a+bx)^{2^k}=a^{2^k}=1\text{ in }\Z_{2^{k+1}}[x]/(x^\sigma-1), 
$$
and the same conclusion holds 
if $a$ is even and $b$ is odd.
This shows that $\pi_\sigma(2^m)\le 2^k$ for $m\le k+1$. As clearly
$\Pi_\sigma(0,1; 2^m)=\sigma=2^k$, we get 
$\pi_\sigma(2^m)=2^k$. 

By the same argument, $(a+bx)^{2^{k+1}}=1$ in 
$\Z_{2^{k+2}}[x]/(x^\sigma-1)$, for all $a$ and $b$. Moreover, 
it is easy to check that $(1+2x)^{2^k}=1+2^{k+1}x+2^{k+1}x^2\ne 1$ in  $\Z_{2^{k+2}}[x]/(x^\sigma-1)$, proving the last claim. 
\end{proof}

Call a prime $p$ \textbf{persistent} if $\pi_\sigma(p)<\ub_\sigma(p)$ for infinitely many $\sigma$. We conclude with 
a few questions 
suggested by Table~\ref{table-pi-ub}.

\begin{ques} 
(1) Is either $2$ or $3$ persistent? 
(2) Are there infinitely many primes $p$ such 
that is $\pi_\sigma(p)<\ub_\sigma(p)$
for some $\sigma$? 
(3) Is $2$ the only prime with 
$\pi_\sigma(p^m)<\ub_\sigma(p^m)$ for some $m\ge 2$?
\end{ques}

\section*{Acknowledgements}
Both authors were partially supported by the NSF grant DMS-1513340.
JG was also supported in part by the Slovenian Research Agency (research program P1-0285). 
The authors thank Jingyang Shu for useful conversations, in particular for pointing out the sequence~(\ref{curious}).

\bibliography{references}

\section{Appendix}
In this appendix, we determine the structure of the 
multiplicative group of Eisenstein integers modulo $n$, that is, 
the group $\mathbb{Z}_n[\omega]^\times=\{
a + b\omega \in \mathbb{Z}_n[\omega]: a^2 + b^2 - ab \in \mathbb{Z}_n^\times\}
$, where $\omega = e^{2\pi i/3}$. 
While our arguments are similar to those in the paper \cite{allan2008classification} on Gaussian integers modulo $n$, 
we are aware of no reference that directly implies Theorem~\ref{theorem: eisenstein}, 
so we provide a sketch of the proof.

\begin{lem}\label{lemma-qr}
1. Let $p\ge 3$ be a prime number and $a$ be an integer not divisible by $p$.
Then $x^2 = a \mod p$ either has no solutions or exactly two solutions.

2. Let $p\ge 5$ be a prime number. The number $-3$ is a quadratic residue modulo $p$ if and only if $p = 1 \mod 6$.
\end{lem}

\begin{proof}
See \cite{allan2008classification} for the proof of part 1.
For part 2, see \cite{leveque1996fundamentals}, Exercise 9 on page 109.
\end{proof}

\begin{lem}\label{lemma-app-p}
Let $p$ be a prime.
1. If $p = 3$, then $\mathbb{Z}_p[\omega]^\times \cong \mathbb{Z}_6$.

2. If $p = 1 \mod 6$, then $\mathbb{Z}_p[\omega]^\times \cong \mathbb{Z}_{p - 1} \times \mathbb{Z}_{p - 1}$.

3. If $p = 5 \mod 6$, then $\mathbb{Z}_p[\omega]^\times \cong \mathbb{Z}_{p^2 - 1}$.
\end{lem}

\begin{proof}
To prove part 1, observe that the group $\mathbb{Z}_3[\omega]^\times$ is abelian,
and $|\mathbb{Z}_3[\omega]^\times| = 6$, 
so $\mathbb{Z}_3[\omega]^\times \cong \mathbb{Z}_6$.

To prove part 2, 
first note that then the equation $x^2 - x + 1 = 0 \mod p$ is equivalent to $(2x-1)^2=-3\mod p$. By Lemma~\ref{lemma-qr}, the equation 
$y^2 = -3 \mod p$, where $y = 2x - 1$  has two solutions $y = \pm q$. 
We next find the cardinality of $\mathbb{Z}_p[\omega]^\times$. 
Assume that $a + b\omega\notin\mathbb{Z}_p[\omega]^\times$, so that $a^2 + b^2 - ab = 0 \mod p$. 
If $a \neq 0 \mod p$, then $(a^{-1}b)^2 - (a^{-1}b) = -1 \mod p$ and so $2a^{-1}b - 1 = \pm q\mod p$.
So, $b =2^{-1} a(\pm q + 1)$.  In particular, 
for a fixed non-zero $a$, there are two possible values for $b$  
such that $a + b\omega\notin \mathbb{Z}_p[\omega]^\times$,
proving that $|\mathbb{Z}_p[\omega]^\times| = (p - 1)^2$.

As $\mathbb{Z}_p^\times \cong \mathbb{Z}_{p - 1}$,
it suffices to show that there is an isomorphism
$$\psi: \mathbb{Z}_p[\omega]^\times \to \mathbb{Z}_p^\times \times \mathbb{Z}_p^\times.$$
It is routine to check that $\psi$, defined by
$\psi(a + b\omega) = 
(a-2^{-1}b(q+1), a - 2^{-1}b (-q + 1))$, is an injective homomorphism, hence it is an isomorphism by equality 
of cardinalities.
 

To prove part 3, note that $\mathbb{Z}_p[\omega]$ has $p^2$ elements, so it suffices to show that $\mathbb{Z}_p[\omega]$ is a field, as the multiplicative group of any field is cyclic.
Assume again that $a + b\omega\notin\mathbb{Z}_p[\omega]^\times$, so that $a^2 + b^2 - ab = 0 \mod p$.
If $a \neq 0\mod p$, then $ (a^{-1}b)^2 - (a^{-1}b) = -1 \mod p$. By Lemma~\ref{lemma-qr}, the equation $x^2 - x + 1 = 0 \mod p$, or equivalently $(2x - 1)^2 = -3 \mod p$, has no solution, as $p = 5 \mod 6$. We conclude that $a = 0\mod p$,
and similarly $b = 0 \mod p$, 
so $\mathbb{Z}_p[\omega]$ is a field.
\end{proof}

\begin{lem} \label{lemma: eisenstein 1}
For a prime $p\ge 3$ and $m \ge 2$,
$$\mathbb{Z}_{p^m}[\omega]^\times \cong \mathbb{Z}_{p^{m - 1}} \times \mathbb{Z}_{p^{m - 1}} \times \mathbb{Z}_{p}[\omega]^\times.$$
\end{lem}
\begin{proof}
The proof is analogous to that for Theorem 7 in \cite{allan2008classification}. 
\end{proof}

\begin{lem} \label{lemma: eisenstein 2} For $m\ge 1$, 
$\mathbb{Z}_{2^m}[\omega]^\times$ is classified as follows:
$\mathbb{Z}_2[\omega]^\times \cong \mathbb{Z}_3$, 
$\mathbb{Z}_{2^2}[\omega]^\times \cong \mathbb{Z}_3 \times \mathbb{Z}_2 \times \mathbb{Z}_2$, and, for
$m \ge 3$, $\mathbb{Z}_{2^m}[\omega]^\times \cong \mathbb{Z}_3 \times \mathbb{Z}_{2^{m - 1}} \times \mathbb{Z}_{2^{m - 2}} \times \mathbb{Z}_2$.
\end{lem}
\begin{proof}
The multiplicative group $\mathbb{Z}_2[\omega]^\times$ is abelian with 3 elements, so $\mathbb{Z}_2[\omega]^\times \cong \mathbb{Z}_3$.
Assume that $m \ge 2$.
Write $H = \mathbb{Z}_{2^m}[\omega]^\times$.
The elements of the group $H$ are of the form $(1 + 2k_1) + 2k_2 \omega$, $2k_1 + (1 + 2k_2)\omega$ and $(1 + 2k_1) + (1 + 2k_2) \omega$  for $0 \le k_1, k_2 \le 2^{m - 1} - 1$, so the number of them is $2^{m - 1}2^{m - 1}3 = 3\times 2^{2m - 2}$. 
Furthermore (see proof of Theorem 7 in \cite{allan2008classification}), each element in $H$ has order at most $3 \cdot 2^{m - 1}$, and 
by verifying that $(1 + 3 \omega)^{3\cdot 2^{m-2}} \neq 1$ in 
$\mathbb{Z}_{2^m}[\omega]$ and $(1 + 3 \omega)^{2^{m-1}} \neq 1$ 
in $\mathbb{Z}_{2^m}[\omega]$, we see that there exists an element 
with order exactly $3 \cdot2^{m - 1}$.
As a consequence, $H \cong \mathbb{Z}_3 \times \mathbb{Z}_{2^{m - 1}} \times \prod_{j = 1}^r \mathbb{Z}_{2^{e_j}}$, where $e_j \ge 1$ and $\sum_{j = 1}^r e_j = m - 1$. When $m = 2$, the result follows immediately, so we assume $m\ge 3$ from now on.

We claim that $r = 2$. Since each factor, except $\mathbb{Z}_3$, is cyclic of order at least two, each contains exactly one subgroup of order two. So, $H$ has $2^{r + 1}$ solutions to the equation $(a + b\omega)^2 = 1\mod 2^m$, which is equivalent to
$$
\begin{cases}
a^2 - b^2 = 1 \mod 2^m\\
2ab - b^2 = 0 \mod 2^m
\end{cases}.
$$
This system has no solution unless $a$ is odd and $b$ is even, 
so we write $a = 2k_1 + 1$ and $b = 2k_2$ and obtain
$$
\begin{cases}
k_1^2 + k_1 - k_2^2 = 0 \mod 2^{m - 2}\\
(2k_1 + 1 - k_2) k_2 = 0 \mod 2^{m - 2}
\end{cases}.
$$
From the first equation, $k_2$ is even, so $2k_1 + 1 - k_2$ has an inverse and then $k_2 = 0 \mod 2^{m - 2}$, so 
$k_2 = 0$ or $2^{m - 2}$. 
Now $k_1(k_1 + 1) = 0 \mod 2^{m - 2}$. 
If $k_1$ is odd, then $k_1 + 1 = 0 \mod 2^{m - 2}$ implies $a = 2^{m - 1} - 1$ or $a = 2^{m} - 1$; if $k_1$ is even, then $k_1 = 0 \mod 2^{m - 2}$ implies $a = 0$ or $a = 2^{m - 1} + 1$.
So, the original system has eight solutions, $2^{r + 1} = 8$ and $r = 2$.

We now have $H \cong \mathbb{Z}_3 \times \mathbb{Z}_{2^{m - 1}} \times \mathbb{Z}_{2^{e_1}} \times \mathbb{Z}_{2^{e_2}}$, where $e_1 + e_2 = m - 1$ and $e_1 \ge e_2$. 
Now, the result follows for $m = 3$ and $4$, so we
assume $m\ge 5$. Then, we claim that 
$e_2 = 1$ and $e_1 = m - 2$.
Assume, to the contrary, that $e_2 \ge 2$.
Then each factor, except $\mathbb{Z}_3$, has exactly one subgroup of order four, giving $4^3 = 64$ elements of order at most four in the direct product. However, 
we will show that $H$ has at most 32 solutions to the equation $x^4 = 1$, which will establish our 
claim and end the proof. To this end, 
suppose $(a + b\omega)^4 = 1$ for some $a + b\omega \in \mathbb{Z}_{2^m}[\omega]$.  
Then 
$$
\begin{cases}
a^4 - 6a^2b^2 + 4ab^3 = 1 \mod 2^m\\
b(4a^3 - 6a^2b^2 + b^3) = 0 \mod 2^m
\end{cases}.
$$
This system has no solutions unless $b$ is even and $a$ is odd, 
so write $a = 2k_1 + 1$ and $b = 2k_2$,  $0 \le k_1, k_2 \le 2^{m - 1} - 1$.
Then the system becomes
$$
\begin{cases}
k_1(k_1 + 1)(2k_1^2 + 2k_1 + 1) - 3(2k_1 + 1)^2k_2^2 + 4(2k_2 + 1)k_2 = 0 \mod 2^{m - 3}\\
k_2\left[(2k_1 + 1)^3 - 6(2k_1 + 1)^2k_2^2 + 2k_2^3 \right] = 0 \mod 2^{m - 3}
\end{cases}.
$$
The factor in square brackets and $2k_1^2 + 2k_1 + 1$ are odd,   reducing the system to 
$$
\begin{cases}
k_1(k_1 + 1) = 0 \mod 2^{m - 3}\\
k_2 = 0 \mod 2^{m - 3}
\end{cases},
$$
which has at most 32 solutions.
\end{proof}

We conclude by summarizing Lemmas~\ref{lemma-app-p}--\ref{lemma: eisenstein 2}.

\begin{thm} \label{theorem: eisenstein}
We have
 $$
\mathbb{Z}_p[\omega]^\times \cong
\begin{cases}
\mathbb{Z}_6, & \text{ if } p = 3\\
\mathbb{Z}_{p - 1} \times \mathbb{Z}_{p - 1}, & \text{ if } p = 1 \mod 3\\
\mathbb{Z}_{p^2 - 1} & \text{ if } p = 2 \mod 3
\end{cases}
$$
and
$$
\mathbb{Z}_{p^m}[\omega]^\times \cong
\begin{cases}
\mathbb{Z}_{p^{m - 1}} \times \mathbb{Z}_{p^{m - 2}} \times \mathbb{Z}_6, & \text{ if } p = 2 \text{ and } m \ge 2\\
\mathbb{Z}_{p^{m - 1}} \times \mathbb{Z}_{p^{m - 1}} \times \mathbb{Z}_p[\omega]^\times, & \text{ if } p\neq 2
\end{cases}.
$$
\end{thm}

\end{document}